\documentclass{amsart}
\pdfoutput=1

\usepackage[utf8]{inputenc}
\usepackage[mathscr]{eucal}
\usepackage{amssymb}
\usepackage{stmaryrd}
\usepackage[usenames,dvipsnames]{xcolor}
\usepackage{xspace}
\usepackage{amsthm}
\usepackage{bbold}
\usepackage[normalem]{ulem}
\usepackage{adjustbox}
\usepackage{enumerate}

\usepackage{array}
\usepackage{amsmath}
\usepackage{wasysym}
\usepackage{etoolbox}

\usepackage{comment}

\numberwithin{equation}{section}
\usepackage[all]{xy}
\SelectTips{cm}{}
\newdir{ >}{{}*!/-10pt/\dir{>}}

\usepackage{xr-hyper}

\usepackage[colorlinks=true,linkcolor={Brown},citecolor={Brown},urlcolor={Brown}]{hyperref}

\usepackage{cleveref}

\crefname{Thm}{Theorem}{Theorems}
\crefname{Rem}{Remark}{Remarks}
\crefname{Prop}{Proposition}{Propositions}

\swapnumbers 
\newtheorem{Cor}[equation]{Corollary}
\newtheorem{Lem}[equation]{Lemma}

\newtheorem{Prop}[equation]{Proposition}

\theoremstyle{remark}
\newtheorem{Def}[equation]{Definition}
\newtheorem{Not}[equation]{Notation}
\newtheorem{Exa}[equation]{Example}

\newtheorem{Hyp}[equation]{Hypothesis}
\newtheorem{Rem}[equation]{Remark}

\newtheorem{Rec}[equation]{Recollection}
\newtheorem{Cons}[equation]{Construction}

\newcommand{\nc}{\newcommand}
\nc{\dmo}{\DeclareMathOperator}

\dmo{\Ab}{Ab}
\dmo{\add}{add}
\dmo{\Add}{Add}
\dmo{\AM}{AM}
\dmo{\ATM}{ATM}
\dmo{\Aut}{Aut}
\dmo{\Chain}{Ch}
\dmo{\Cpx}{\Chain}
\dmo{\coker}{coker}
\dmo{\cone}{cone}
\dmo{\Defl}{Defl}
\dmo{\Der}{D}
\dmo{\DAM}{DAM^{\geom}}%
\dmo{\DAMbig}{DAM}%
\dmo{\DATM}{DATM^{\geom}}%
\dmo{\DATMbig}{DATM}%
\dmo{\DTM}{DTM^{\geom}}%
\dmo{\eff}{eff}
\dmo{\Ext}{Ext}
\dmo{\Gal}{Gal}
\dmo{\gr}{gr}
\dmo{\grmodname}{grmod}%
\dmo{\Hm}{H}
\dmo{\Hom}{Hom}
\dmo{\Hty}{\Kb}
\dmo{\Id}{Id}
\dmo{\incl}{incl}
\dmo{\ind}{ind}
\dmo{\Ind}{Ind}
\dmo{\Infl}{Infl}
\dmo{\IInfl}{IInfl} 
\dmo{\Ker}{Ker}
\dmo{\MAc}{\mathcal{E}_{\textup{c}}}
\dmo{\MA}{\mathcal{E}}
\dmo{\modname}{mod}%
\dmo{\Mod}{\mathsf{Mod}}
\dmo{\mot}{M}
\dmo{\noeth}{noeth}
\dmo{\opname}{op}
\dmo{\perm}{\mathsf{perm}}
\dmo{\Perm}{\mathsf{Perm}}
\dmo{\Qcoh}{Qcoh}
\dmo{\quot}{quot}
\dmo{\RMod}{RMod}
\dmo{\RProj}{RProj}
\dmo{\RInj}{RInj}
\dmo{\sgn}{sgn}
\dmo{\SH}{SH^{\mathrm{c}}}
\dmo{\SHmot}{SH^{\mathrm{c}}_{\bbA^{\!1}}}
\dmo{\smallperf}{perf}
\dmo{\Span}{span}
\dmo{\Spc}{Spc}
\dmo{\Spec}{Spec}
\dmo{\Spech}{Spec^h}
\dmo{\st}{st}
\dmo{\stab}{stab}
\dmo{\Stab}{Stab}
\dmo{\sing}{sing}
\dmo{\StabD}{\mathsf{sing}}
\dmo{\supp}{supp}
\dmo{\fpname}{fp}
\dmo{\yoneda}{h}
\dmo{\thick}{thick}
\dmo{\triv}{triv}
\dmo{\TM}{TM}
\dmo{\Locname}{Loc}

\nc{\ababs}{{\sl ab absurdo}\xspace}
\nc{\AKB}{\bat{A}(\cK;\cat{B})}%
\nc{\ala}{{\`a la}\ }
\nc{\ac}{\mathrm{ac}}
\nc{\Ac}{\ac}
\nc{\atau}[1][\tau]{%
        \mathsf{a}_{#1}}
\nc{\aet}{\atau[\et]}
\nc{\atr}{\mathsf{a}_{\mathrm{tr}}}
\nc{\CohMack}[1][]{\mathop{\mathsf{Mack}^{\mathsf{coh}}_{#1}}}
\nc{\Mack}[1][]{\mathop{\mathsf{Mack}_{#1}}}
\nc{\Gac}{G\textrm{-}\ac}
\nc{\Gaac}{\Gamma\textrm{-}\ac}
\nc{\SAc}{\Gac}
\nc{\Inj}{\mathrm{Inj}}
\nc{\injres}{\mathbb{J}}
\nc{\calO}{\mathcal{O}}
\nc{\cat}[1]{\mathscr{#1}}
\nc{\caT}[1]{\cat{#1}}
\nc{\cc}{\mathsf{c}}
\nc{\cA}{\cat{A}}
\nc{\cB}{\cat{B}}
\nc{\cE}{\cat{E}}
\nc{\cF}{\cat{F}}
\nc{\cG}{\cat{G}}
\nc{\cH}{\cat{H}}
\nc{\cI}{\cat{I}}
\nc{\cJ}{\cat{J}}
\nc{\cK}{\cat{K}}
\nc{\cL}{\cat{L}}
\nc{\cM}{\cat{M}}
\nc{\cN}{\cat{N}}
\nc{\colim}{\mathop{\mathrm{colim}}}
\nc{\hocolim}{\mathop{\mathrm{hocolim}}}
\nc{\cP}{\cat{P}}
\nc{\cQ}{\cat{Q}}
\nc{\cR}{\cat{R}}
\nc{\cS}{\cat{S}}
\nc{\cT}{\cat{T}}
\nc{\cV}{\cat{V}}
\nc{\cX}{\cat{X}}
\nc{\cY}{\cat{Y}}
\nc{\Dperf}{\Der_{\smallperf}}
\nc{\Dsing}{\Der^{\sing}}
\nc{\StD}{\Dsing}
\nc{\DRperf}[1][]{\Der_{\kern-0.1em\ifblank{#1}{\CR}{#1}\text{-}\kern-0.1em\smallperf}}
\nc{\eg}{{\sl e.g.}\@\xspace}
\nc{\FP}{\mathrm{FP}} 
\nc{\FPL}{\mathrm{LP}} 
\nc{\grmod}[1]{#1\text{-}\kern-0.1em\grmodname}%
\nc{\fp}{^{\fpname}}
\nc{\gp}{\mathfrak{p}}
\nc{\Homcat}[1]{\Hom_{\cat #1}}
\nc{\hook}{\hookrightarrow}
\nc{\ie}{{\sl i.e.}\@\xspace}
\nc{\into}{\mathop{\rightarrowtail}}
\nc{\inv}{^{-1}}
\nc{\kk}{k}
\nc{\kkG}{\kk G}
\dmo{\CInvname}{\chi}
\nc{\CInv}[2]{\CInvname^{#1}} 
\nc{\CInvGR}{\CInv{G}{\CR}}
\nc{\CInvHR}{\CInv{H}{\CR}}
\nc{\CINVGR}{\bar{\CInvname}^{G}_{\CR}}
\nc{\Loc}[1]{\Locname(#1)}
\nc{\Locab}[1]{\langle#1\rangle}
\nc{\Locat}[1]{\Locab{\cat{#1}}}
\nc{\loccit}{{\sl loc.\ cit.}\xspace}
\nc{\Mid}{\,\big|\,}
\nc{\mmod}[1]{\modname(#1)}%
\nc{\MMod}[1]{\Mod(#1)}%
\nc{\onto}{\mathop{\twoheadrightarrow}}
\nc{\op}{^{\opname}}
\nc{\otr}{\mathsf{o}_{\mathrm{tr}}}
\nc{\sminus}{\smallsetminus}
\nc{\oursetminus}{\smallsetminus}
\nc{\pos}{\mathrm{pos}}
\nc{\potimes}[1]{^{\otimes #1}}
\nc{\permGR}{\perm(G;\CR)}
\nc{\pperm}{p\textrm{-}\!\perm}
\nc{\ppermutation}{$\natural$-permutation\xspace}
\nc{\sbull}{{\scriptscriptstyle\bullet}}
\nc{\brk}[1]{\{#1\}}
\nc{\SET}[2]{\big\{\,#1\Mid#2\,\big\}}
\nc{\PSh}[2]{\mathrm{PSh}_{#1}(#2)}
\nc{\Sh}[3]{\mathrm{Sh}_{#1}(#2;#3)}
\nc{\sstab}{\stabname\kern-0.1em\text{-}}%
\nc{\too}{\mathop{\longrightarrow}\limits}
\nc{\unit}{\mathbb{1}}
\nc{\via}{via\xspace}
\nc{\vcorrect}[1]{{\vphantom{\vbox to #1em{}}}}

\nc{\Keff}{\cK^{\eff}}%
\dmo{\DPerm}{D\mathsf{Perm}}
\nc{\DPermGR}{\DPerm(G;\CR)}
\nc{\Aname}{\cat{P}}
\nc{\Bname}{\cat{Q}}
\nc{\Acat}[2]{\Aname(#1;#2)}
\nc{\ACAT}[2]{\K\Inj_{\perm}(#1;#2)}
\nc{\ACATGR}{\ACAT{G}{\CR}}
\nc{\Bcat}[2]{\Bname(#1;#2)}
\nc{\AcatGk}{\Acat{G}{\kk}}
\nc{\BcatGk}{\Bcat{G}{\kk}}
\nc{\AcatGR}{\Acat{G}{\CR}}
\nc{\BcatGR}{\Bcat{G}{\CR}}
\nc{\CP}[2]{\cat{CP}(#1;#2)}

\nc{\Permtomod}{\Upsilon}
\nc{\barPermtomod}{\bar{\Permtomod}}
\nc{\F}{\Permtomod}
\nc{\barF}{\barPermtomod}
\nc{\tbarF}{U}
\nc{\Fplus}{\F^{\scriptscriptstyle+}}
\nc{\barFplus}{\barF^{\scriptscriptstyle+}}

\nc{\ConPerm}[2]{\Der_{\perm}(#1;#2)}
\nc{\Dperm}[2]{\ConPerm{#1}{#2}}
\nc{\ConPermGR}{\ConPerm{G}{\CR}}
\nc{\DpermGR}{\Dperm{G}{\CR}}

\nc{\KG}[2]{\rmG_0(#2#1)}
\nc{\KP}[2]{\rmK_0(#2#1)}
\nc{\KA}[2]{\rmK_0^{\Aname}(#1;#2)}
\nc{\KB}[2]{\rmK_0^{\Bname}(#1;#2)}
\nc{\KC}[2]{C_{#2}(#1)}
\nc{\AGk}{\KA{G}{\kk}}
\nc{\AGR}{\KA{G}{\CR}}
\nc{\BGk}{\KB{G}{\kk}}
\nc{\BGR}{\KB{G}{\CR}}
\nc{\CGk}{\KC{G}{\kk}}
\nc{\RGk}{\KG{G}{\kk}}
\nc{\RGR}{\KG{G}{\CR}}
\nc{\PGk}{\KP{G}{\kk}}
\nc{\PGR}{\KP{G}{\CR}}

\nc{\fundpur}{\mathrm{S}}
\nc{\G}{\mathbb{g}}
\nc{\I}{\mathbb{i}}
\nc{\invertpur}{\mathrm{L}}
\nc{\Sierp}{\Big\{\,\vcenter{\xymatrix@R=1em@H=.5em{*={\bullet}\ar@{-}[d]\\*={\circ}}}\,\Big\}}
\nc{\V}{\mathbb{V}}

\dmo{\quo}{quo}
\dmo{\sta}{sta}
\dmo{\Sta}{Sta}
\dmo{\fgt}{fgt}
\nc{\rsd}[1]{\mathrm{rsd}_{#1}}
\nc{\Reet}{\mathrm{Re}_{\et}}

\dmo{\End}{End}
\nc{\isoto}{\overset{\sim}{\,\to\,}}
\nc{\isofrom}{\overset{\sim}{\,\leftarrow\,}}
\nc{\xisoto}[1]{\xrightarrow[\sim]{#1}}
\nc{\xto}[1]{\xrightarrow{#1}}
\nc{\xfrom}[1]{\xleftarrow{#1}}
\nc{\xinto}[1]{\overset{#1}{\,\into\,}}
\nc{\xonto}[1]{\overset{#1}{\,\onto\,}}
\nc{\lto}{\leftarrow}
\nc{\qquadtext}[1]{\qquad\textrm{#1}\qquad}
\nc{\quadtext}[1]{\quad\textrm{#1}\quad}
\def\tobar{\mathrel{\mkern3mu  \vcenter{\hbox{$\scriptscriptstyle+$}}%
                    \mkern-12mu{\to}}}
\nc{\normal}{\vartriangleleft}

\nc{\lecl}{\Subset}

\dmo{\can}{can}
\dmo{\chara}{char}%
\dmo{\Chow}{Chow}
\nc{\Corr}{\mathsf{Cor}}
\dmo{\corr}{Cor} 
\dmo{\DMbig}{DM}
\dmo{\DMminuseff}{DM^{\,--,\mathrm{eff}}}
\dmo{\DMbigeff}{DM^{\mathrm{eff}}}
\dmo{\DMetbig}{DM_{\et}}
\dmo{\DMetbigeff}{DM_{\et}^{\mathrm{eff}}}
\dmo{\DM}{DM^{\geom}}
\dmo{\DMeff}{DM^{\geom,\mathrm{eff}}}
\dmo{\ev}{ev}
\dmo{\Gl}{Gl}
\dmo{\id}{id}
\dmo{\Img}{Im}
\dmo{\im}{im}
\dmo{\Komp}{K}
\dmo{\Map}{Map}%
\dmo{\orb}{or}
\dmo{\Orb}{Or}
\dmo{\proj}{proj}
\dmo{\Pic}{Pic}
\dmo{\Proj}{Proj} 
\dmo{\rmG}{G}
\dmo{\rmH}{H}
\nc{\rmh}{\mathsf{h}}
\dmo{\rmK}{K}
\nc{\rmL}{\mathsf{L}}
\nc{\rmR}{\mathsf{R}} 
\dmo{\Res}{Res}
\dmo{\res}{res}
\dmo{\smallb}{b}
\dmo{\geom}{gm}
\dmo{\Schm}{Sch}
\dmo{\stabname}{stab}
\dmo{\Sm}{\mathsf{Sm}}
\dmo{\Tor}{Tor}
\dmo{\Kos}{Kos}

\nc{\AbGrps}{\MMod{\bbZ}}
\nc{\adh}[1]{\overline{#1}}
\nc{\adhoc}{{\sl ad hoc}\xspace}
\nc{\adhpt}[1]{\adh{\{#1\}}}
\nc{\adj}{\dashv}
\nc{\adjto}{\rightleftarrows}
\nc{\afortiori}{{\sl a fortiori}}
\nc{\aka}{{a.\,k.\,a.}\ }
\nc{\apriori}{{\sl a priori}\xspace}
\nc{\bs}{\backslash}
\nc{\cf}{{\sl cf.}\ }
\nc{\Cb}{\Chain_{\smallb}}
\nc{\Ch}{\Cb}
\nc{\CR}{R}
\nc{\CRG}{\CR G}
\nc{\Db}{\Der_{\smallb}}
\nc{\Dbs}{\Db^{\sing}}
\nc{\D}{\Der}
\nc{\eps}{\epsilon}
\nc{\equalby}[1]{\overset{\textrm{#1}}=}
\nc{\et}{\textup{\'et}}
\nc{\Et}{\textup{\'Et}}
\nc{\etal}{{\sl et al.}}
\nc{\FFsep}{\overline{\FF}}
\nc{\FFq}{\FF_{\!q}}
\nc{\Fp}{\bbF_{\!p}}
\nc{\gm}{\mathfrak{m}}
\mathchardef\mhyphen="2D
\nc{\Gasets}[1][]{\ifblank{#1}{\Gamma}{#1}\mathsf{\mhyphen Sets}}
\nc{\gasets}[1][]{\ifblank{#1}{\Gamma}{#1}\mathsf{\mhyphen sets}}
\nc{\ideal}[1]{\langle #1\rangle}
\nc{\idealK}[1]{\langle #1\rangle_{\scriptscriptstyle\cK}}
\nc{\ihom}{{\mathsf{hom}}} 
\nc{\Kb}{\Komp_{\smallb}}
\nc{\K}{\Komp}
\nc{\Kbac}{\Komp_{\smallb,\mathrm{ac}}}
\nc{\KKac}{\KK_{\mathrm{ac}}}
\nc{\Kac}{\K_{\mathrm{ac}}}
\nc{\Kp}{\Komp_{+}}
\nc{\Km}{\Komp_{-}}
\nc{\Ksp}{\Komp_{+,\mathrm{sp}}}
\nc{\Kbm}{\Komp_{\mathrm{b},\le0}}
\nc{\Lotimes}{\otimes^{\rmL}}
\nc{\Nis}{\textup{Nis}}
\nc{\pproj}[1]{#1\text{-}\kern-0.1em\proj}%
\nc{\Pout}[1]{\Paul{\sout{#1}}}
\nc{\restr}[1]{_{|_{\scriptstyle #1}}}
\nc{\smat}[1]{\left(\begin{smallmatrix} #1 \end{smallmatrix}\right)}
\nc{\Sn}[1]{\mathfrak{S}_{#1}}
\nc{\Snm}{\Sn{m}}
\nc{\Spccat}[1]{\Spc(\cat #1)}
\nc{\SpcH}{\Spc(\cH)}
\nc{\SpcK}{\Spc(\cK)}
\nc{\To}{\Rightarrow}
\nc{\Top}{\mathsf{Top}}
\nc{\EndHere}{\bibliographystyle{alpha}\bibliography{ref}
\nc{\Hg}{K\cap{}^{g\!}H}
\nc{\Kg}{K^{\!g}\cap H}

\date{\today}

\author{Paul Balmer}
\address{Paul Balmer, UCLA Mathematics Department, Los Angeles, CA 90095-1555, USA}
\email{balmer@math.ucla.edu}
\urladdr{https://www.math.ucla.edu/~balmer}

\author{Martin Gallauer}
\address{Martin Gallauer, Oxford Mathematical Institute, Oxford, OX2 6GG, UK}
\email{gallauer@maths.ox.ac.uk}
\urladdr{https://people.maths.ox.ac.uk/gallauer}

\hypersetup{pdfauthor={Paul Balmer, Martin Gallauer},pdftitle={Permutation modules, Mackey functors, and Artin motives}}
\externaldocument[pmcs:]{pmcs}[https://arxiv.org/pdf/2009.14093.pdf]

\usepackage{tikz}
\usetikzlibrary{arrows,shapes,positioning,backgrounds,cd}
\tikzset{-,>=stealth',shorten >=2pt,shorten <=2pt,
  main node/.style={rectangle,fill=blue!5,font=\scshape},
  label/.style={font=\itshape},
}

\begin{document}
\makeatletter\@input{pmcs.tex}\makeatother

\title[Permutation modules, Mackey functors, and Artin motives]{Permutation modules, Mackey functors,\\ and Artin motives}

\begin{abstract}
We explain in detail the connections between the three concepts in the title, and we discuss how the `big' derived category of permutation modules introduced in~\cite{balmer-gallauer:resol-big} fits into the picture.
\end{abstract}

\subjclass[2010]{}
\keywords{Artin motive, permutation module, Mackey functor, derived category}

\thanks{First-named author supported by NSF grant~DMS-1901696.}

\maketitle


\section{Introduction}
\label{sec:intro}%

\subsection*{Objectives}

This article is a companion to our work in progress on Artin-Tate motives from the point of view of tensor-triangular geometry, which we have started to document in~\cite{balmer-gallauer:rage,balmer-gallauer:resol-small,balmer-gallauer:resol-big}.
As such, its goal is to explain in detail the beautiful connections between that subject matter and certain representation theoretic topics.
Partly, these connections are used in the work alluded to in order to infer from representation theory to algebraic geometry, and partly they make it possible to deduce consequences for representation theoretic questions from our results in algebraic geometry.

In a nutshell, the main topics to be discussed are depicted in \Cref{fig:main-triangle}, with the names of those mathematicians who arguably contributed the most to our understanding of the corresponding interrelations between these topics.
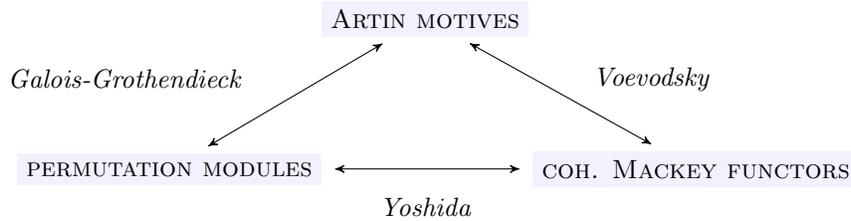
\begin{figure}[h]
\centering
\begin{tikzpicture}[shorten >=5pt,shorten <=5pt]
\node[main node] at (0,1) (A) {Artin motives};
\node[main node] at (-3.4,-1) (P) {permutation modules};
\node[main node] at (3.6,-1) (M) {coh. Mackey functors};
\node[label] at (-4,0.2) {Galois-Grothendieck};
\node[label] at (3,0.2) {Voevodsky};
\node[label] at (0,-1.5) {Yoshida};
\path[<->]
(A) edge (M)
(M) edge (P)
(P) edge (A);
\end{tikzpicture}
\caption{Diagrammatic representation of main topics}
\label{fig:main-triangle}
\end{figure}
In their most basic form, these connections are classical.
\emph{Grothendieck's Galois theory} describes an equivalence between finite \'etale-algebras over a field $\FF$ and finite sets with an action by the absolute Galois group $\Gamma=\Gamma_\FF$.
\textsc{Permutation modules} are obtained by linearizing $\Gamma$-sets, and the corresponding process under the Galois equivalence results precisely in \'etale correspondences, the zero-dimensional analogue of smooth correspondences that are so central in \emph{Voevodsky's} approach to motives.
This connects sheaves with transfers (and thus eventually \textsc{Artin motives}) and additive presheaves on permutation modules.
Finally, the latter are nothing but \textsc{cohomological Mackey functors}, as observed by \emph{Yoshida}.

Our interest is ultimately in these objects at the derived level, and more precisely in their `big derived categories' (that is, compactly generated triangulated categories).
Let us fix a commutative ring $\CR$.
As the category of $\CR$-linear cohomological Mackey functors for $\Gamma$, denoted $\CohMack[\CR]{\Gamma}$, is abelian, deriving it poses no difficulty.
Similarly, Voevodsky's motives are intrinsically derived, and therefore so are his Artin motives $\DAMbig(\FF;\CR)$.
However, permutation modules form an additive category, and it does not seem obvious what its `derived' category should be.
In~\cite{balmer-gallauer:resol-big}, we proposed a candidate, denoted $\DPerm(\Gamma;\CR)$, and discussed it briefly.
Our second goal in this article then is to discuss this candidate in more detail and to justify our proposal.
In particular, we will give a motivated definition which is intrinsic to permutation modules, establish fundamental properties of the resulting category, and relate it to Artin motives and cohomological Mackey functors.\,(\footnote{\,Due to space limitations, we do not discuss in this article a fourth incarnation of $\DPerm(\Gamma;\CR)$. To wit, in the context of equivariant homotopy theory it may also be identified with the homotopy category of the $\infty$-category $\Mod_{\mathrm{H}\underline{\CR}}(\mathrm{SH}(\Gamma))$. Here, $\mathrm{SH}(\Gamma)$ is the stable homotopy category of (genuine) $\Gamma$-equivariant spectra, and $\mathrm{H}\underline{\CR}$ is the Eilenberg-Maclane ring spectrum associated to the constant Mackey functor with value~$\CR$.})

\subsection*{Content}

We now turn to the contents of this paper in more detail, and we start by recalling our proposal for the derived category of permutation modules.
Throughout, $\Gamma$ is a profinite group and $\CR$ is a commutative ring.

To motivate the definition of $\DPerm(\Gamma;\CR)$, let us take a page from equivariant homotopy theory.
If $G$ is a finite group, then a $G$-equivariant map $f\colon X\to Y$ between $G$-spaces $X$ and~$Y$ is called a $G$-weak-equivalence if the induced maps on fixed points \mbox{$f^H\colon X^H\to Y^H$} are weak equivalences, for all subgroups~$H\le G$.
A conceptual reason is that orbits~$G/H$ are `equivariant points' and $\Map_G(G/H,X)\cong X^H$.
Transposing to representation theory, where $\Hom_{\MMod{\Gamma;\CR}}(\CR(\Gamma/H),M)\cong M^H$, we define \emph{$\Gamma$-quasi-isomorphisms} as those morphisms $f\colon M\to N$ of complexes of discrete $(\Gamma;\CR)$-modules such that $f^H\colon M^H\to N^H$ is a quasi-isomorphism for all open subgroups~$H\le \Gamma$.

\begin{Def}
\label{Def:DPerm-intro}%
The \emph{derived category of permutation modules} is the Verdier localization
\[
\DPerm(\Gamma;\CR)=\K(\Perm(\Gamma;\CR))\big[\{\Gamma\textrm{-quasi-isos}\}\inv\big]
\]
of the homotopy category of (not necessarily finitely generated) permutation modules with respect to $\Gamma$-quasi-isomorphisms.
\end{Def}

Among other things, we will establish the following:
\begin{enumerate}[\rm(a)]
\item
\label{it:DPerm-properties}%
The category $\DPerm(\Gamma;\CR)$ is tensor triangulated, compactly generated, and its compact part is the idempotent-completion of $\Kb(\perm(\Gamma;\CR))$, the bounded homotopy category of finitely generated permutation modules.
\item
\label{it:connections}%
If $\Gamma$ is the absolute Galois group of a field $\FF$, then $\DPerm(\Gamma;\CR)$ fits into the following diagram of tensor-triangulated equivalences:
\begin{equation}
\label{eq:triangle-derived}%
\vcenter{\xymatrix{
  &
  \DAMbig(\FF;\CR)
  \ar@{<->}[rd]^\sim
  \ar@{<->}[ld]_\sim
  \\
  \DPerm(\Gamma;\CR)
  \ar@{<->}[rr]_{\sim}
  &&
  \D(\CohMack[\CR]{\Gamma})
}}
\end{equation}
This realizes the interrelations set out in \Cref{fig:main-triangle} in a precise form.
Note also that $\DPerm(\Gamma;\CR)$ is therefore a \textsl{bona fide} derived category of an abelian category.
In particular, it admits a t-structure -- which, alas, does not restrict to the compact part.
\end{enumerate}

As for the structure of the document, in \Cref{sec:discrete-perm} we recall some basics on discrete and permutation modules, and in \Cref{sec:DPerm} we define $\DPerm(\Gamma;\CR)$ and establish property~\eqref{it:DPerm-properties}.
Much of the article (\Crefrange{sec:mackey-functors}{sec:motivic}) is devoted to the connections mentioned at the beginning of this introduction and depicted in \Cref{fig:main-triangle}, culminating in property~\eqref{it:connections}.
Finally, since the Mackey functoriality plays such an important role in this story, it would be a shame not to discuss the higher-level Mackey ($2$-)functoriality which the three categories in~\eqref{eq:triangle-derived} exhibit in the argument $\Gamma$ and $\FF$, respectively.
We do this in \Cref{sec:mackey-functoriality}, and establish that the equivalences in~\eqref{eq:triangle-derived} are in spirit equivalences of Mackey 2-functors in the sense of~\cite{balmer-dellambrogio:two-mackey}.
\subsection*{Intended audience}

Most, if not all, results in this article have appeared in the literature in one form or another, although not always in the present generality nor from the present point of view.
Instead, our intention was to produce a unified treatment of topics which are strongly related but stretch across different disciplines.
To cater for readers with various backgrounds, we have tried to be more thorough and elementary than it would otherwise have been necessary.
We therefore hope that, at least, algebraic geometers might learn something from representation theory, and conversely, representation theorists something from algebraic geometry.


\subsection*{Notation and conventions}
\phantomsection
\label{sec:conventions}

Throughout we fix a ring $\CR$ of coefficients, commutative and with unit.
Also, $\Gamma$ denotes a profinite group.
In \Cref{sec:Shtr,sec:motivic} it will often be the absolute Galois group of some field with a fixed separable algebraic closure.

All $\CR$-linear categories and functors are implicitly assumed to be additive.
A tensor category is an additive category with a symmetric monoidal structure, additive in both variables.
Similarly, an $\CR$-linear tensor category is $\CR$-linear and a tensor category such that the tensor product is $\CR$-linear in both variables.

Given an additive category $\cA$, we denote by $\cA^\natural$ its idempotent-completion (also called the Karoubi envelope).


\section{Discrete and permutation modules}
\label{sec:discrete-perm}%

In this section we are going to recall basic facts about discrete $(\Gamma;\CR)$-modules and, most relevantly for us, the subcategory of permutation modules.

\begin{Rec}
\label{Rec:G-sets}%
Recall the category $\Gasets$, whose objects are sets (viewed as discrete topological spaces) on which $\Gamma$ acts continuously, and whose morphisms are $\Gamma$-equivariant maps.
Continuity of the action on a set $X$ is equivalent to every stabilizer subgroup $\Gamma_x:=\SET{\gamma\in\Gamma}{\gamma \cdot x=x}$ being open, \ie closed and of finite index, for every~$x\in X$.
The category $\Gasets$ has arbitrary coproducts, given by disjoint union.
We may thus write every $X\in\Gasets$ as a coproduct of transitive $\Gamma$-sets.
By continuity, each transitive $\Gamma$-set is finite, (non-canonically) isomorphic to $\Gamma/H$ for some open subgroup $H\le \Gamma$.

The Cartesian product of sets on which $\Gamma$ acts diagonally, endows $\Gasets$ with a symmetric monoidal structure, and the Cartesian product commutes with arbitrary coproducts.

We will also be interested in the full subcategory $\gasets$ of finite $\Gamma$-sets.
The symmetric monoidal structure restricts to $\gasets$.
\end{Rec}

\begin{Rec}
\label{Rec:discrete-modules}%
Consider the category $\MMod{\Gamma;\CR}$ of discrete $(\Gamma;\CR)$-modules, that is, $\CR$-modules endowed with the discrete topology, on which $\Gamma$ acts continuously.
Continuity of the action on an $\CR$-module~$M$ is equivalent to every stabilizer subgroup $\Gamma_m:=\SET{\gamma\in\Gamma}{\gamma \cdot m=m}$ being open.

The tensor product over $\CR$ with the diagonal $\Gamma$-action endows $\MMod{\Gamma;\CR}$ with a tensor structure:
\[
M\otimes_{\CR}N,\qquad g(m\otimes n)=gm\otimes gn.
\]
Similarly, the $\CR$-module $\Hom_{\CR}(M,N)$ admits an `(anti)diagonal' $\Gamma$-action defined by $(g\,f)(m)=g\,f(g\inv m)$. The case $N=\CR$ with trivial $\Gamma$-action yields the $\CR$-linear dual $M^*=\Hom_{\CR}(M,\CR)$.
\end{Rec}

\begin{Not}
In this article we will not consider anything but \emph{discrete} $(\Gamma;\CR)$-modules, and we will therefore often omit the adjective.

The notation $H\le \Gamma$ will always denote an \emph{open} subgroup of $\Gamma$.
Similarly, $N\normal \Gamma$ will denote an \emph{open} normal subgroup of $\Gamma$.
\end{Not}

\begin{Not}
\label{Not:perm}%
Let $X$ be a $\Gamma$-set.
We denote the associated discrete $(\Gamma,\CR)$-module by $\CR(X)$.
In other words, $\CR(X)$ is the free $\CR$-module on the basis $X$, and the $\Gamma$-action is $\CR$-linearly extended from the action on $X$.
Thus a functor
\begin{equation}
\CR(-):\Gasets\to\MMod{\Gamma;\CR}\label{eq:CR(-)}
\end{equation}
from $\Gamma$-sets to discrete modules, that preserves coproducts, and sends finite products to tensor products, in other words, is symmetric monoidal.
An object in the essential image of this functor is called a \emph{permutation module}, and we denote the full subcategory on permutation modules by $\Perm(\Gamma;\CR)$.

The \emph{finitely generated permutation modules} are those in the essential image of $\gasets$ under $\CR(-)$.
They span a full subcategory denoted by $\perm(\Gamma;\CR)$.
Thus every object in $\Perm(\Gamma;\CR)$ (respectively, $\perm(\Gamma;\CR)$) is a (respectively, finite) direct sum of permutation modules of the form $\CR(\Gamma/H)$, $H\le \Gamma$.
Both are tensor categories (\cf our conventions set out on \cpageref{sec:conventions}).
\end{Not}

Recall that a family $\cG$ of objects in an abelian category is \emph{generating} if the functor $\prod_{g\in\cG}\Hom(g,-)$ is faithful.
Having a generating set is a necessary condition for an abelian category to be Grothendieck.
Recall also that an object $M$ in a Grothendieck abelian category is called \emph{finitely presented} if every map $M\to\colim_i M_i$ into a filtered colimit factors through some $M_i$.
\begin{Prop}
\label{discrete-modules-grothendieck}%
The category $\MMod{\Gamma;\CR}$ is Grothendieck abelian.
A family of finitely presented generators is given by $\SET{\CR(\Gamma/H)}{H\le\Gamma\text{ open}}$.
\end{Prop}
\begin{proof}
One easily verifies that the (conservative) forgetful functor $\MMod{\Gamma;\CR}\to\MMod{\CR}$ creates finite limits and finite colimits as well as filtered colimits.
In other words, those constructions, performed in $\CR$-modules still admit a continuous $\Gamma$-action.
Then $\MMod{\Gamma;\CR}$ admits filtered colimits (hence coproducts), and filtered colimits preserve monomorphisms.
Since $\SET{\CR(\Gamma/H)}{H\le \Gamma\textrm{ open}}$ is a family of generators, $\MMod{\Gamma;\CR}$ is a Grothendieck category.
That the objects in this family are finitely presented is clear.
\end{proof}

\begin{Rem}
\label{Rem:perm-dual}%
The objects in $\perm(\Gamma;\CR)$ are self-dual with respect to the tensor product.
Explicitly, if $X$ is a finite $\Gamma$-set viewed as an $\CR$-basis for $\CR(X)$, and if we denote the dual basis of $\CR(X)^*=\Hom_{\CR}(R(X),R)$ by $(\delta_x)_{x\in X}$, then the map
\begin{equation}
  \label{eq:perm-self-dual}
  \begin{aligned}
  \CR(X)&\to\CR(X)^*\\
  x&\mapsto \delta_x
  \end{aligned}
\end{equation}
is $\Gamma$-equivariant, and therefore an isomorphism in $\perm(\Gamma;\CR)$.
\end{Rem}

\begin{Lem}
\label{perm-dual-morphism}%
Let $f:X\to Y$ be a morphism in $\gasets$, and the associated morphism $f:\CR(X)\to\CR(Y)$ in $\perm(\Gamma;\CR)$.
Its dual admits the following explicit description:
\begin{align*}
  \CR(Y)\stackrel{\text{\eqref{eq:perm-self-dual}}}{\simeq}\CR(Y)^*&\xto{f^*} \CR(X)^*\stackrel{\text{\eqref{eq:perm-self-dual}}}{\simeq}\CR(X)\\
  y&\mapsto\sum_{f(x)=y}x.
\end{align*}
\end{Lem}
\begin{proof}
This is straightforward verification, using \Cref{Rem:perm-dual}.
\end{proof}

\begin{Rec}
\label{Rec:Mackey-operations-discrete}%
Let $\Gamma'\le\Gamma$ be an open subgroup and $\gamma\in\Gamma$.
Given $M\in\MMod{\Gamma;\CR}$ we may restrict the action to obtain $\Res^\Gamma_{\Gamma'}M\in\MMod{\Gamma';\CR}$.
This defines a \emph{restriction} functor which has an adjoint (on both sides) called \emph{induction}:
\[
\Ind^\Gamma_{\Gamma'}:\MMod{\Gamma';\CR}\rightleftarrows\MMod{\Gamma;\CR}:\Res^\Gamma_{\Gamma'}
\]
It is easy to see that they restrict to an adjunction on permutation modules.

Given $M\in\MMod{\Gamma';\CR}$ we denote by $c_{\gamma}(M)$ the same $\CR$-module on which $g\in{}^\gamma\Gamma'=\gamma\Gamma'\gamma\inv$ acts \via $\gamma\inv g\gamma$.
This defines an adjoint equivalence (isomorphism)
\[
c_{\gamma}:\MMod{\Gamma';\CR}\simeq\MMod{{}^\gamma\Gamma';\CR}:c_{\gamma\inv}
\]
that restricts to permutation modules.
The $c_?$ are called \emph{conjugation} functors.
\end{Rec}

These functors satisfy axioms reminiscent of Mackey functors (see \Cref{sec:mackey-functors}) and we will say more about this in \Cref{sec:mackey-functoriality}.
Here, we recall only the most substantial of these axioms, namely the Mackey formula.
\begin{Rem}
\label{Rem:Mackey-formula}%
Fix two (open, as always) subgroups $H,K\le\Gamma$.
The Mackey formula is a natural (but non-canonical) isomorphism of functors $\MMod{H;\CR}\to \MMod{K;\CR}$
\begin{equation*}
\Res^\Gamma_K\circ\Ind^\Gamma_H\simeq \oplus_{[g]\in K\backslash{}\Gamma/H}\ \Ind^K_{\Hg}\circ c_{g}\circ\Res^H_{\Kg}
\end{equation*}
where we write $[g]\in K\backslash{}\Gamma/H$ to mean that we have fixed a choice of a representative $g\in \gamma$ for every class $\gamma\in K\backslash{}\Gamma/H$.
From this, one deduces a Mackey formula for tensor products:
\begin{equation*}
\Ind^\Gamma_K\CR\otimes \Ind^\Gamma_HW
\simeq\Ind^\Gamma_K
\left(
\Res^\Gamma_K\Ind^\Gamma_HW
\right)
\simeq\oplus_{[g]\in K\backslash{}\Gamma/H}\Ind^\Gamma_{\Hg}c_{g}\Res^H_{\Kg}W.
\end{equation*}
In particular, for $W=\CR$, we get
\begin{equation}
\label{eq:mackey-formula-tensor-special}
\CR(\Gamma/K)\otimes\CR(\Gamma/H)\simeq\oplus_{[g]\in K\backslash{}\Gamma/H}\,\CR(\Gamma/(\Hg))
\end{equation}
where the isomorphism may be explicitly described on the canonical $\CR$-basis as
\[
[\gamma]_K\otimes [\gamma g]_H\mapsfrom [\gamma]_{\Hg}\,.
\]
\end{Rem}

\begin{Cor}
\label{perm-hom-sets}%
Fix $H,K\le \Gamma$.
There is an $\CR$-linear isomorphism
\[
\CR(K\backslash\Gamma/H)\simeq\Hom_{\perm(\Gamma;\CR)}(\CR(\Gamma/K),\CR(\Gamma/H)).
\]
Explicitly, the isomorphism sends $[g]$ to the map $([\gamma]_K\mapsto \sum_{[x]\in K/\Hg} [\gamma xg]_H)$.
\end{Cor}
\begin{proof}
By \Cref{Rem:perm-dual}, morphisms $\CR(\Gamma/K)\to \CR(\Gamma/H)$ are in bijection with morphisms
\[
\CR\to\CR(\Gamma/K)\otimes\CR(\Gamma/H)\stackrel{\text{\eqref{eq:mackey-formula-tensor-special}}}{\simeq}\oplus_{[g]\in K\backslash{}\Gamma/H}\CR(\Gamma/\Hg),
\]
that is, with $\Gamma$-invariant elements of $\oplus_{[g]\in K\backslash{}\Gamma/H}\CR(\Gamma/\Hg)$.
Thus the first claim.
The explicit description of this isomorphism follows from the explicit description of~\eqref{eq:mackey-formula-tensor-special} in \Cref{Rem:Mackey-formula}.
\end{proof}

\begin{Rem}
\label{Rem:Inflation}%
Let $N\subseteq \Gamma$ be a \emph{closed but not necessarily open} normal subgroup and set $\bar{\Gamma}=\Gamma/N$.
Restriction along the quotient map $\Gamma\to\bar{\Gamma}$ induces the inflation functor
\[
\Infl_{\bar{\Gamma}}^{\Gamma}:\MMod{\bar\Gamma;\CR}\to\MMod{\Gamma;\CR}.
\]
For later use we record the following fact.
\end{Rem}

\begin{Lem}
\label{Lem:Infl-fully-faithful}%
The inflation functor is fully faithful on permutation modules:
\[
\Infl_{\bar{\Gamma}}^{\Gamma}:\perm(\bar\Gamma;\CR)\into\perm(\Gamma;\CR).
\]
\end{Lem}
\begin{proof}
This follows immediately from \Cref{perm-hom-sets}. Indeed, if $N\subset K,H\le \Gamma$ and if we set $\bar{H}=H/N$, $\bar{K}=K/N$, then the map on hom sets induced by inflation
\[
\Hom_{\perm(\bar\Gamma;\CR)}(\CR(\bar\Gamma/\bar K),\CR(\bar\Gamma/\bar H))\to \Hom_{\perm(\Gamma;\CR)}(\CR(\Gamma/K),\CR(\Gamma/H))
\]
corresponds to the isomorphism $\CR(\bar K\backslash\bar\Gamma/\bar H)\cong\CR(K\backslash\Gamma/H)$ of double cosets.
\end{proof}

\section{The derived category of permutation modules}
\label{sec:DPerm}%

In this section we define one of the central objects in this paper, the derived category of permutation modules, and establish some of its fundamental properties.

\begin{Def}
\label{Def:Gamma-acyclics}%
An object $X\in\K(\Perm(\Gamma;\CR))$ is said to be \emph{$\Gamma$-acyclic} if for each open subgroup $H\le \Gamma$, the complex $X^H$ of $H$-fixed points is acyclic.
A morphism $f$ in $\K(\Perm(\Gamma;\CR))$ is said to be a \emph{$\Gamma$-quasi-isomorphism} if its cone is $\Gamma$-acyclic.
The full subcategory of $\Gamma$-acyclic objects is denoted by $\K_{\Gaac}(\Perm(\Gamma;\CR))$, and the class of $\Gamma$-quasi-isomorphisms is denoted by $QI_\Gamma$.
\end{Def}
\begin{Rem}
\label{Rem:Gamma-qis}%
Of course, $f:X\to Y$ is a $\Gamma$-quasi-isomorphism if and only if, for each open $H\le\Gamma$, the morphism $f^H:X^H\to Y^H$ is a quasi-isomorphism.
\end{Rem}

\begin{Rem}
\label{Rem:Gamma-acyclics}%
For all~$X\in\K(\Perm(\Gamma;\CR))$, $H\le \Gamma$, and $n\in\bbZ$, we have
\begin{equation*}
\Hom_{\K(\MMod{\Gamma;\CR})}(\CR(\Gamma/H)[n],X)=\Hom_{\K(\MMod{\CR})}(\CR[n],X^H)=\Hm_n(X^H).
\end{equation*}
Hence we obtain the identification
\begin{equation}
\label{eq:Gaac-orthogonal}%
\K_{\Gaac}(\Perm(\Gamma;\CR))=\{\CR(\Gamma/H)\mid H\le \Gamma\}^\perp=\left(\Kb(\perm(\Gamma;\CR))\right)^\perp
\end{equation}
with the right orthogonal complement of finitely generated permutation modules in~$\K(\Perm(\Gamma;\CR))$.
\end{Rem}

\begin{Rem}
\label{Rem:DPerm-vs-KPerm}%
In general, $\K_{\Gaac}(\Perm(\Gamma;\CR))$ is non-zero.
For example, let $\Gamma=C_p=\langle \sigma\mid \sigma^p=1\rangle$ be a cyclic group of odd prime order, let $\CR=\kk$ be a field of characteristic $p$, and consider the complex
\[
\xymatrix@C=3em{
  X=
  &
  \cdots
  \ar[r]^-{\smat{\sigma-1&\eta\\\epsilon&0}}
  &
  kC_p\oplus k
  \ar[r]^{\smat{\alpha&\eta\\\epsilon&0}}
  &
  kC_p\oplus k
  \ar[r]^{\smat{\sigma-1&\eta\\\epsilon&0}}
  &
  kC_p\oplus k
  \ar[r]^-{\smat{\alpha&\eta\\\epsilon&0}}
  &
  \cdots,
}
\]
where $\eta:k\to kC_p$ and $\epsilon:kC_p\to k$ are the unit and counit, respectively, and where $\alpha$ is multiplication by $\sum_{i=0}^{p-1} i\sigma^i$.
It is acyclic and stays so after taking $C_p$-fixed points.
On the other hand, $X$ is not contractible. Indeed, the inclusion $K=\ker\smat{\sigma-1&\eta\\\epsilon&0}\into kC_p\oplus k$ cannot split in $\MMod{C_p;k}$ by Krull-Schmidt, because the $k$-vector space~$K$ is 2-dimensional and $p>2$.
\end{Rem}

\begin{Def}
\label{Def:DPerm}%
The \emph{(big) derived category of permutation modules} of $\Gamma$ with coefficients in $\CR$ is the Verdier localization:
\[
\DPerm(\Gamma;\CR)=\K(\Perm(\Gamma;\CR))[QI_\Gamma\inv]=\frac{\K(\Perm(\Gamma;\CR))}{\K_{\Gaac}(\Perm(\Gamma;\CR))}.
\]
\end{Def}

All we know \apriori is that $\DPerm(\Gamma;\CR)$ is a triangulated category.
Before establishing some of its additional fundamental structures and properties we need to recall general facts from~\cite{neeman:thomason-localization}; see~\cite[Chapter~9]{neeman:book-tricats} or~\cite{krause:localization-theory-tricat}.

\begin{Rec}
\label{Rec:compact-perfect}%
Let $\cS$ be a triangulated category with all (small) coproducts.
An object $X\in\cS$ is called compact if $\Hom_{\cS}(X,-)$ preserves coproducts.
The subcategory of compact objects is denoted by $\cS^c$.
And $\cS$ is compactly generated if there is a set of compact objects $\cG\subset\cS^c$ which generates $\cS$ as a localizing subcategory. The latter condition is equivalent to the right orthogonal complement $\cG^\perp=\{Y\in\cS\mid \Hom(\Sigma^nX,Y)=0, \forall n\in\bbZ,X\in\cG\}$ being zero.
\end{Rec}
\begin{Rec}
\label{Rec:Neeman-loc}%
Let $\cS$ and~$\cT$ be triangulated categories with all coproducts.
\begin{enumerate}[\rm(a)]
\item
\label{it:Neeman-Brown}%
Brown Representability:
If $\cS$ is compactly generated, an exact functor \mbox{$F\colon\cS\to \cT$} preserves coproducts if and only if it admits a right adjoint.
A left adjoint preserves compacts if and only if its right adjoint preserves coproducts.
\item
\label{it:Neeman-compacts}%
\label{it:Neeman-sub-quotient}%
Neeman-Thomason Localization:
Let $\cG\subseteq\cT^c$ be a set of compacts objects. Then $\cS=\Loc{\cG}$ is compactly generated and the inclusion $\cS\into \cT$ admits a right adjoint~$\Delta$, by~\eqref{it:Neeman-Brown}.
Moreover $\cS^c=\thick(\cG)$ and the composite $\cS\into\cT\onto\cT/\cS^\perp$ is an equivalence, with quasi-inverse induced by~$\Delta$.
\item
\label{it:Neeman-recollement}%
In the situation of \eqref{it:Neeman-compacts}, if moreover~$\cT$ is compactly generated, that right adjoint~$\Delta$ admits another right adjoint and we have a recollement of triangulated categories
\[
\xymatrix@R=2em{
\cS
\ar@{ >->}@/_2pc/[d]_-{\incl}
\ar@{ >->}@/^2pc/[d]_{}
\\
\cT
\ar@{->>}@/_2pc/[d]_{L}
\ar@{->>}@/^2pc/[d]^{}
\ar@{->>}[u]^{\Delta}
\\
\cS^\perp
\ar@{ >->}[u]^{\incl}
}
\]
\end{enumerate}
\end{Rec}

Note that $\K(\Perm(\Gamma;\CR))$ admits small coproducts. We may therefore consider the localizing subcategory generated by permutation modules.
\begin{Prop}
\label{DPerm-as-subcategory}%
The composite
\[
\Loc{\CR(\Gamma/H)\mid H\le\Gamma}\into \K(\Perm(\Gamma;\CR))\onto\DPerm(\Gamma;\CR)
\]
is an equivalence.
\end{Prop}
\begin{proof}
Since homology and the fixed-point functors preserve coproducts, it follows from \Cref{Rem:Gamma-acyclics} that the $\CR(\Gamma/H)$ are compact objects in $\K(\Perm(\Gamma;\CR))$, for all $H\le \Gamma$.
The result then follows directly from \Cref{Rec:Neeman-loc}\,\eqref{it:Neeman-sub-quotient} for $\cG:=\SET{\CR(\Gamma/H)}{H\le\Gamma}$ and $\cT:=\K(\Perm(\Gamma;\CR))$.
Indeed, by~\eqref{eq:Gaac-orthogonal}, the right orthogonal of $\cS=\Loc{\cG}$ is $\cS^\perp=\K_{\Gaac}(\Perm(\Gamma;\CR))$.
\end{proof}

\begin{Cor}
\label{Cor:DPerm-compacts}%
The triangulated category $\DPerm(\Gamma;\CR)$ is compactly generated and its subcategory of compact objects is canonically equivalent to the thick subcategory of $\K(\Perm(\Gamma;\CR))$ generated by permutation modules:
\begin{equation}
\label{eq:DPerm-compacts}%
\DPerm(\Gamma;\CR)^c\cong\thick(\perm(\Gamma;\CR))=\Kb(\perm(\Gamma;\CR)^{\natural})
\end{equation}
\end{Cor}
\begin{proof}
This follows from \Cref{DPerm-as-subcategory}, by \Cref{Rec:Neeman-loc}\,\eqref{it:Neeman-compacts}.
\end{proof}

\begin{Rem}
\label{Rem:DPerm-switch}%
In~\cite[\Cref{pmcs:Rem:DPerm}]{balmer-gallauer:resol-big} we \emph{defined} $\DPerm(\Gamma;\CR)$ as the localizing subcategory of $\K(\Perm(\Gamma;\CR))$ generated by (transitive) permutation modules.
\Cref{DPerm-as-subcategory} shows that this definition is compatible with ours, and from now on, we will switch freely between viewing $\DPerm(\Gamma;\CR)$ as a quotient or a subcategory of $\K(\Perm(\Gamma;\CR))$.
\end{Rem}

Since $\Perm(\Gamma;\CR)$ is a tensor category (\Cref{Not:perm}) and the tensor product commutes with coproducts, the category $\K(\Perm(\Gamma;\CR))$ is tensor triangulated and the tensor product commutes with coproducts.
We now observe that this tensor structure restricts to $\DPerm(\Gamma;\CR)$.
\begin{Cor}
\label{DPerm-tensor}%
The subcategory $\DPerm(\Gamma;\CR)\subseteq\K(\Perm(\Gamma;\CR))$ is closed under tensor products and thereby inherits a tensor triangulated structure.
In particular, the tensor product commutes with small coproducts.
\end{Cor}
\begin{proof}
It suffices to observe that the subcategory of compact objects $\Kb(\perm(\Gamma;\CR)^\natural)$ is closed under tensor products in $\K(\Perm(\Gamma;\CR))$.
\end{proof}

\begin{Rem}
\label{Rem:DPerm-Bousfield}%
The quotient functor $\cat{K}\onto\cat{K}[QI_\Gamma\inv]$ realizes $\DPerm(\Gamma;\CR)$ as a Bousfield \emph{co}localization of~$\cat{K}=\K(\Perm(\Gamma;\CR))$.
We will later see (\Cref{Rem:DPerm-KPerm-recollement}) that it is also a Bousfield localization (\ie that quotient also admits a fully faithful right adjoint), at least if $\Gamma$ is finite.
This will become more transparent once we translate the question into the language of cohomological Mackey functors as we start doing in the next section.
\end{Rem}

We end this section with the `derived' analogue of \Cref{Lem:Infl-fully-faithful}.
\begin{Lem}
\label{Lem:Infl-fully-faithful-DPerm}%
Let $N\subseteq\Gamma$ be a \emph{closed but not necessarily open} normal subgroup and set $\bar\Gamma=\Gamma/N$.
Then inflation induces a fully faithful embedding
\[
\Infl_{\bar{\Gamma}}^{\Gamma}:\DPerm(\bar\Gamma;\CR)\to\DPerm(\Gamma;\CR)
\]
with image the localizing subcategory generated by $\CR(\Gamma/K)$ where $N\subseteq K\le\Gamma$.
\end{Lem}
\begin{proof}
Inflation induces a functor on compact objects (by \Cref{Cor:DPerm-compacts}):
\[
\Infl_{\bar{\Gamma}}^{\Gamma}:\Kb(\perm(\bar\Gamma;\CR)^\natural)\to\Kb(\perm(\Gamma;\CR)^\natural).
\]
This functor is fully faithful, by \Cref{Lem:Infl-fully-faithful}.
The claim follows immediately.
\end{proof}

\begin{Rem}
\label{Rem:Kbperm-union}%
On compact objects, the fully faithfulness of inflation allows us to reduce many questions to the case where $\Gamma$ is a finite group.
Indeed, it follows that
\[
\Kb(\perm(\Gamma;\CR)^\natural)=\bigcup_{N\normal\Gamma}\Kb(\perm(\Gamma/N;\CR)^\natural).
\]
\end{Rem}

\section{Mackey functors}
\label{sec:mackey-functors}

\begin{Rec}
\label{Rec:mackey-functors}%
Mackey functors for a finite group $G$ may be defined in several different ways, and we refer to~\cite{thevenaz-webb:structure-mackey} or~\cite{dellambrogio:1-mackey-revisited} for ample details.
In an elementary way, Mackey functors associate $\CR$-modules to subgroups of $G$, together with $\CR$-linear induction, restriction, and conjugation actions which satisfy a list of very sensible axioms, including the ``Mackey axiom'' which stipulates that Mackey's formula should hold (and gives the Mackey functors their name).

Equivalently, Mackey functors may be viewed as ``bifunctors''
\[
(M_*,M^*):\gasets[G]\to\MMod{\CR}
\]
on the category of finite $G$-sets, where the covariant part $M_*$ encodes induction and the contravariant part $M^*$ encodes restriction (and they both encode conjugation).
These bifunctors are meant to satisfy two axioms: additivity and, again, the Mackey axiom which can now be expressed as a base change formula for commuting the covariant and the contravariant part.

Both of these definitions readily generalize to profinite groups, and they are again shown to be equivalent in~\cite[Theorem~2.7]{bley-boltje:comack-profinite}.
Note that in the formulation in terms of bivariant functors, Mackey functors are perhaps better known as ``$\Gamma$-modulations''~\cite[Definition~1.5.10]{neukirch-schmidt-wingberg:cohomology-of-number-fields}.
\end{Rec}

\begin{Rec}
\label{correspondences-profinite}%
The most convenient description of Mackey functors for us is the following.
It is a special case of Mackey functors on compact closed categories~\cite{panchadcharam-ross:mackey-functors}.
Recall (\Cref{Rec:G-sets}) that $\gasets$ denotes the category of finite $\Gamma$-sets, and consider the associated category of spans, denoted $\Span(\gasets)$.
In other words, the objects are the same as those of $\gasets$, and a morphism from $X$ to $Y$ is an isomorphism class of spans
\[
\xymatrix@R=1em{
  &
  Z
  \ar[ld]
  \ar[rd]
  \\
  X
  &&
  Y
}
\]
where two spans $(X,Z,Y)$ and $(X,Z',Y)$ are considered isomorphic if there is an isomorphism $Z\isoto Z'$ making the two obvious triangles commute.
The composite of two spans $(X,Z,Y)$ and $(Y,V,W)$ is obtained by forming the Cartesian square
\[
\xymatrix@R=1em{
  &&
  Z\times_YV
  \ar[ld]
  \ar[rd]
  \\
  &
  Z
  \ar[ld]
  \ar[rd]
  &&
  V
  \ar[ld]
  \ar[rd]
  \\
  X
  &&
  Y
  &&
  W.
  }
\]
We sometimes denote a span from $X$ to $Y$ by the symbol $X\tobar Y$.

Note that if $\phi:\Gamma\to \Gamma'$ is a morphism of pro-finite groups, there is a canonical functor $\gasets[\Gamma']\to\gasets$ by restricting along $\phi$, and an induced functor $\Span(\gasets[\Gamma'])\to\Span(\gasets)$.
If $\phi$ is surjective, we will typically denote the induced functor by $\Infl_{\Gamma'}^\Gamma$.
\end{Rec}

\begin{Lem}
Let $N\subseteq\Gamma$ be a closed but not necessarily open normal subgroup.
The functor
\[
\Infl_{\Gamma/N}^{\Gamma}:\Span(\gasets[\Gamma/N])\to\Span(\gasets[\Gamma])
\]
is faithful.
\end{Lem}
\begin{proof}
The functor $\Infl_{\Gamma/N}^{\Gamma}:\gasets[\Gamma/N]\to\gasets[\Gamma]$ is fully faithful from which the claim immediately follows.
\end{proof}

\begin{Cor}
\label{Span-union}%
The category $\Span(\gasets)$ is the filtered colimit of the (possibly non-full) subcategories $\Span(\gasets[\Gamma/N])$ for normal open subgroups $N\normal \Gamma$.
\qed
\end{Cor}

\begin{Def}
\label{Def:profinite-Mackey-functor}%
The disjoint union endows $\Span(\gasets)$ with both finite coproducts and finite products.
It follows that group completing the hom-sets (which are commutative monoids) makes the category additive.
We denote the resulting additive category by
\[
\Omega(\Gamma).
\]

An \emph{($\CR$-linear) Mackey functor} is an additive functor $\Omega(\Gamma)\op\to\MMod{\CR}$.
It is called \emph{cohomological} if, in addition, it sends the span $\Gamma/H\xfrom{\pi}\Gamma/K\xto{\pi}\Gamma/H$ to multiplication by the index $[H:K]$ whenever $K\le H$, and where $\pi$ denotes the canonical projection.
A morphism of (cohomological) Mackey functors is a natural transformation of functors.
This defines abelian categories
\[
\CohMack[\CR]{\Gamma}\subseteq\Mack[\CR]{\Gamma}.
\]

Note that we may extend scalars to $\CR$ formally, denoting the resulting category by $\Omega_\CR(\Gamma)$, in which case Mackey functors are identified with $\CR$-linear (additive) functors $\Omega_\CR(\Gamma)\op\to\MMod{\CR}$.
\end{Def}

\begin{Rem}
\label{Rem:Mackey-comparison-bifunctor}%
There are canonical faithful embeddings
\[
\xymatrix@R=.3em{
  \gasets
  \ar[r]
  &
  \Span(\gasets)
  &
  (\gasets)\op
  \ar[l]
  \\
  (f:X\to Y)
  \ar@{|->}[r]
  &
  (X=X\xto{f}Y)
  \\
  &
  (Y\xfrom{f}X=X)
  &
  (f:X\to Y)
  \ar@{|->}[l]
}
\]
Composing with these embeddings (followed by the canonical functor to $\Omega(\Gamma)$) we obtain from every Mackey functor $M:\Omega(\Gamma)\op\to\MMod{\CR}$ a bifunctor $(M_*,M^*)$ on $\gasets$.
For $\Gamma$ finite, this identifies Mackey functors in the `new' sense of \Cref{Def:profinite-Mackey-functor} with Mackey functors in the `old' sense of \Cref{Rec:mackey-functors}~\cite{lindner:mackey-functors}.
It now follows formally from \Cref{Span-union} that the same is true for pro-finite groups $\Gamma$.
\end{Rem}

\begin{Rem}
\label{Rem:spans-properties}%
There is an auto-duality $\Omega(\Gamma)\op\simeq\Omega(\Gamma)$ which sends a span $X\lto Z\to Y$ to $Y\lto Z\to X$ (and is the identity on objects).
In particular, we could also define a Mackey functor to be an additive functor $\Omega(\Gamma)\to\MMod{\CR}$.

The symmetric monoidal structure on $\gasets$ (which is the Cartesian product on underlying sets, see \Cref{Rec:G-sets}) induces a tensor structure on $\Omega(\Gamma)$.
\end{Rem}

\begin{Lem}
\label{Omega-to-perm}%
Let $\CR(-):\gasets\to\perm(\Gamma;\CR)$ be the functor of~\eqref{eq:CR(-)}, and let $\CR(-)^*:\gasets\op\to\perm(\Gamma;\CR)$ be that functor followed by taking duals.
There is an essentially unique (necessarily additive) functor $\CR(-):\Omega(\Gamma)\to\perm(\Gamma;\CR)$ such that the following diagram commutes up to isomorphisms (the horizontal arrows being the canonical embeddings, \Cref{Rem:Mackey-comparison-bifunctor}):
\[
\xymatrix{
  \gasets
  \ar[rd]_{\CR(-)}
  \ar[r]
  &
  \Omega(\Gamma)
  \ar[d]^{\CR(-)}
  &
  \gasets\op
  \ar[ld]^{\CR(-)^*}
  \ar[l]
  \\
  &
  \perm(\Gamma;\CR)
  }
\]
\end{Lem}
\begin{proof}
As the category $\Omega(\Gamma)$ is obtained from $\Span(\gasets)$ by group-completing the homomorphism monoids, and since $\perm(\Gamma;\CR)$ is additive, it is clear that it suffices to prove the same statement with $\Span(\gasets)$ instead of $\Omega(\Gamma)$. In that case, we can verify the universal property of the category of spans~\cite[Prop.\,A.5.3]{balmer-dellambrogio:two-mackey}.

The extension $\CR(-):\Span(\gasets)\to\perm(\Gamma;\CR)$ is given by
\begin{align*}
  X&\mapsto \CR(X)\simeq\CR(X)^* &&\text{(\Cref{Rem:perm-dual})}\\
  (X\xfrom{f}Z\xto{g}Y) &\mapsto \left(\CR(X)\simeq\CR(X)^*\xto{f^*}\CR(Z)^*\simeq\CR(Z)\xto{g}\CR(Y)\right).
\end{align*}
It suffices to show the following property.
Given a Cartesian square in $\gasets$:
\[
\xymatrix{
  X
  \ar[r]^{\delta}
  \ar[d]_\gamma
  &
  Y
  \ar[d]^\alpha
  \\
  W
  \ar[r]_\beta
  &
  Z
}
\]
we want to show that $\alpha^*\circ \beta=\delta\circ\gamma^*$ as morphisms $\CR(W)\to\CR(Y)$ in $\perm(\Gamma;\CR)$.
Using \Cref{perm-dual-morphism}, we recognize the two morphisms as sending $w\in W$ to
\[
\sum_{y\in Y\,\mid\,\alpha(y)=\beta(w)}y\qquad \text{ and }\qquad \sum_{x\in X\,\mid\,\gamma(x)=w}\delta(x)
\]
respectively.
The fact that the square was Cartesian to start with implies that these two elements in $\CR(Y)$ are equal.
\end{proof}

\begin{Rem}
\label{Rem:R-on-morphisms}%
We record for later the following explicit description of the functor $\CR(-):\Omega(\Gamma)\to\perm(\Gamma;\CR)$ on spans such as
\begin{equation}
\label{eq:generic-span}
\vcenter{\xymatrix@R=.5em{
  &
  Z
  \ar[ld]_f
  \ar[rd]^g
  \\
  X
  &&
  Y.
}}
\end{equation}
It follows from \Cref{Omega-to-perm} and \Cref{perm-dual-morphism} that~\eqref{eq:generic-span} gets sent to the $\CR$-linear map $\CR(X)\to\CR(Y)$ extending
\[
x\longmapsto \sum_{z\in Z\,\mid\, f(z)=x}g(z).
\]
\end{Rem}

\begin{Rec}
\label{Rec:Omega-hom-sets}%
We will need the following description of the monoids of homomorphisms in $\Span(\gasets)$ from~\cite[Proposition~2.2]{thevenaz-webb:structure-mackey}.
Let $H,K\le \Gamma$.
Then $\Hom_{\Span(\gasets)}(\Gamma/K,\Gamma/H)$ is the free abelian monoid on the basis represented by diagrams
\begin{equation}
\label{eq:Omega-basis}
\vcenter{\xymatrix@R=.5em{
  &
  \Gamma/L
  \ar[dl]_{{}^g\pi^K_L}
  \ar[rd]^{\pi^H_{L}}
  \\
  \Gamma/K
  &
  &
  \Gamma/H
}}
\end{equation}
where $[g]\in K\backslash{}\Gamma/H$ and $L\le \Kg$ up to $\Kg$-conjugacy. The maps ${}^g\pi^K_L$ and $\pi^H_L$ are the obvious (twisted) projections sending $[\gamma]_L$ to  $[{}^g\gamma]_K$ and $[\gamma]_H$ respectively.
It follows that the group $\Hom_{\Omega(\Gamma)}(\Gamma/K,\Gamma/H)$ is free abelian on the same basis.

Note that by \Cref{Rem:R-on-morphisms}, the image of the span~\eqref{eq:Omega-basis} in $\perm(\Gamma;\CR)$ is the $\CR$-linear map $\CR(\Gamma/K)\to\CR(\Gamma/H)$ extending
\begin{equation}
\label{eq:R-on-basis}%
[\gamma]_K\longmapsto \sum_{[x]\in K/{}^{g\!}L}[\gamma xg]_H.
\end{equation}
\end{Rec}

\begin{Rem}
\label{Rem:span-factorization}%
On pro-finite groups, an important difference between Mackey functors and cohomological Mackey functors is that the latter are more strongly controlled by their values on finite quotients.
To make this precise, we start with the following elementary observation.
Consider a prototypical span~\eqref{eq:Omega-basis}.
We may rewrite it as
\[
\xymatrix@R=1em{
  &\Gamma/L
  \ar[d]^{\pi^{\Kg}_L}
  \\
  &\Gamma/\Kg
  \ar[ld]_{{}^g\pi^K_{\Kg}}
  \ar[rd]^{\pi_{\Kg}^H}
  \\
  \Gamma/K
  &&
  \Gamma/H
}
\]
and thereby recognize that it factors as a composite:
\[
\xymatrix@C=1.2em@R=.5em{
  &
  \Gamma/\Kg
  \ar@{=}[rd]
  \ar[ld]_-{{}^g\pi^K_{\Kg}}
  &
  &\Gamma/L
  \ar[ld]_-{\pi_L^{\Kg}}
  \ar[rd]^-{\pi_L^{\Kg}}
  &&
  \Gamma/\Kg
  \ar@{=}[ld]
  \ar[rd]^-{\pi^H_{\Kg}}
  \\
  \Gamma/K
  &&
  \Gamma/\Kg
  &&
  \Gamma/\Kg
  &&
  \Gamma/H.
}
\]
Recall that the span in the middle is taken to multiplication by $[\Kg:L]$ by each \emph{cohomological} Mackey functor.
In other words, the values a cohomological Mackey functor takes on spans $\Gamma/K\tobar\Gamma/H$, are determined by its restriction to $\Gamma/N$ for any $N\normal\Gamma$ contained in $H\cap K$.
\end{Rem}

\begin{Not}
\label{Not:Omega-ideal}%
Consider the additive category $\Omega_{\CR}(\Gamma)$.
We define a two-sided ideal of homomorphisms~$\cI_{\CR}(\Gamma)$ by giving a set of generators:
\begin{equation}
\label{eq:Omega-ideal-generators}%
\begin{split}
\xymatrix@R=.5em{
  &
  \Gamma/K
  \ar[ld]_{\pi^H_K}
  \ar[rd]^{\pi^H_K}
  \\
  \Gamma/H
  &&
  \Gamma/H
}
\end{split}
\quad - \quad [H:K]\cdot\id_{\Gamma/H}
\end{equation}
for each $K\le H\le\Gamma$.
\end{Not}

\begin{Prop}
\label{Omega-perm-quotient}%
The functor $\CR(-):\Omega_{\CR}(\Gamma)\to\perm(\Gamma;\CR)$ of \Cref{Omega-to-perm} induces an equivalence of additive categories
\begin{equation}
\label{eq:Omega-perm-quotient}%
\frac{\Omega_{\CR}(\Gamma)}{\cI_{\CR}(\Gamma)}\isoto\perm(\Gamma;\CR).
\end{equation}

\end{Prop}
\begin{proof}
It is clear from~\eqref{eq:R-on-basis} that every generator~\eqref{eq:Omega-ideal-generators} is sent to 0 in $\perm(\Gamma;\CR)$.
As the functor $\CR(-)$ is additive, this yields the functor~\eqref{eq:Omega-perm-quotient} of the statement.

It is also clear that the functor is essentially surjective, and it remains to prove fully faithfulness.
Fix $H,K\le \Gamma$ and consider the induced homomorphism
\begin{equation}
\label{eq:Omega-perm-homs}%
\Hom_{\Omega_{\CR}(\Gamma)}(\Gamma/K,\Gamma/H)\to\Hom_{\perm(\Gamma;\CR)}(\CR(\Gamma/K),\CR(\Gamma/H)).
\end{equation}
The codomain of this homomorphism is isomorphic to $\CR(K\backslash{}\Gamma/H)$ by \Cref{perm-hom-sets}, and from~\eqref{eq:R-on-basis} we deduce that the span
\begin{equation}
\vcenter{\xymatrix@C=3em@R=.5em{
  &
  \Gamma/\Kg
  \ar[dl]_-{{}^g\pi^K_{\Kg}}
  \ar[dr]^-{\pi^H_{\Kg}}
  \\
  \Gamma/K
  &&
  \Gamma/H
}}\label{eq:special-span}
\end{equation}
gets sent to $[g]\in \CR(K\backslash{}\Gamma/H)$.
In other words, we have found an $\CR$-free submodule $M$ (on the basis of spans as in~\eqref{eq:special-span}, \cf \Cref{Rec:Omega-hom-sets}) of $\Hom_{\Omega_{\CR}(\Gamma)}(\Gamma/K,\Gamma/H)$ which maps isomorphically onto $\Hom_{\perm(\Gamma;\CR)}(\CR(\Gamma/K),\CR(\Gamma/H))$.
On the other hand, we saw in \Cref{Rem:span-factorization} that every span $\Gamma/K\tobar\Gamma/H$ is equivalent, modulo $\cI_{\CR}(\Gamma)$, to a linear combination of spans as in~(\ref{eq:special-span}).
This shows that the composite
\[
M\into \Hom_{\Omega_{\CR}(\Gamma)}(\Gamma/K,\Gamma/H)\onto\frac{\Hom_{\Omega_{\CR}(\Gamma)}(\Gamma/K,\Gamma/H)}{\cI_{\CR}(\Gamma)(\Gamma/K,\Gamma/H)}
\]
is also surjective and hence an isomorphism, and this concludes the proof.
\end{proof}

\begin{Cor}
\label{coh-mackey-perm}%
Let $M:\Omega(\Gamma)\op\to\MMod{\CR}$ be a Mackey functor and consider the following solid-arrows diagram:
\[
\xymatrix{
  \Omega(\Gamma)\op
  \ar[r]^M
  \ar[d]_{\CR(-)}
  &
  \MMod{\CR}
  \\
  \perm(\Gamma;\CR)\op
  \ar@{.>}[ru]_{\bar{M}}
}
\]
If an $\CR$-linear factorization $\bar{M}$ as indicated exists then it is unique.
Moreover, it does exist if and only if $M$ is cohomological.
\end{Cor}
\begin{proof}
Both $\CR(-)$ and $M$ factor uniquely through $\Omega_{\CR}(\Gamma)\op$, and we may translate the diagram and the statement into one with $\Omega(\Gamma)\op$ replaced by $\Omega_{\CR}(\Gamma)\op$.
By \Cref{Omega-perm-quotient}, $\perm(\Gamma;\CR)$ is a quotient of $\Omega_{\CR}(\Gamma)$ from which the first statement follows.
It also follows from \Cref{Omega-perm-quotient} that $\bar{M}$ exists if and only if $M$ sends generators~(\ref{eq:Omega-ideal-generators}) of $\cI_{\CR}(\Gamma)$ to 0 which is precisely the condition that $M$ be cohomological.
\end{proof}

Let us also state explicitly the following immediate consequence, essentially due to Yoshida~\cite{yoshida:G-functors}.
\begin{Cor}
\label{Cor:comack-presh-perm}%
Precomposition with $\CR(-):\Omega(\Gamma)\to\perm(\Gamma;\CR)$ induces an equivalence of $\CR$-linear Grothendieck abelian categories
\[
\CohMack[\CR]{\Gamma}\simeq\PSh{\CR}{\perm(\Gamma;\CR)}
\]
where we write $\PSh{\CR}{\perm(\Gamma;\CR)}$ for the category of $\CR$-linear (additive) presheaves $\perm(\Gamma;\CR)\op\to\MMod{\CR}$, \aka the category of right-modules over the additive category~$\perm(\Gamma;\CR)$.
\qed
\end{Cor}

\begin{Rem}
\label{Rem:coh-mackey-tensor}%
\Cref{Cor:comack-presh-perm} allows us to define a tensor structure on cohomological Mackey functors.
For this, recall that $\perm(\Gamma;\CR)$ is an $\CR$-linear tensor category.
There is an essentially unique way of endowing $\CR$-linear presheaves with an $\CR$-linear closed tensor structure such that the Yoneda functor $\perm(\Gamma;\CR)\to\PSh{\CR}{\perm(\Gamma;\CR)}$ is tensor, called the Day convolution product.
We will from now on view $\CohMack[\CR]{\Gamma}$ as an $\CR$-linear tensor category \via \Cref{Cor:comack-presh-perm} and Day convolution.

Concretely, given a cohomological Mackey functor $M$ viewed as an $\CR$-linear presheaf on $\perm(\Gamma;\CR)$, we may write it canonically as a colimit of representables, indexed by the category of permutation modules~$P$ over $M$:
\[
\colim_{P\to M}P\isoto M.
\]
In particular, we find that
\[
M\otimes M'\cong \colim_{P\to M,\,P'\to M'}P\otimes P'.
\]
\end{Rem}

\section{The functor of fixed-points}
\label{sec:perm-vs-comack}%

We now want to establish one of the equivalences of~\eqref{eq:triangle-derived}, namely
\begin{equation}
\label{eq:perm-cohmack}%
\DPerm(\Gamma;\CR)\isoto \D(\CohMack[\CR]{\Gamma})
\end{equation}
between the derived category of permutation modules $\DPerm(\Gamma;\CR)$ introduced in \Cref{sec:DPerm} and the derived category of the abelian category $\CohMack[\CR]{\Gamma}$ of cohomological Mackey functors discussed in \Cref{sec:mackey-functors}. To this end, we identify $\CohMack[\CR]{\Gamma}$ with the category of contravariant $\CR$-linear functors from $\perm(\Gamma;\CR)$ to $\MMod{\CR}$, as in \Cref{Cor:comack-presh-perm}.
Let us start by defining the functor in~\eqref{eq:perm-cohmack}.
\begin{Cons}
\label{Cons:FP}%
Note that $\perm(\Gamma;\CR)$ is a subcategory of~$\MMod{\Gamma;\CR}$. Hence any discrete module $M\in\MMod{\Gamma;\CR}$ defines by `restricted Yoneda' an $\CR$-linear functor
\[
\FP(M)=\Hom_{\MMod{\Gamma;\CR}}(-,M)\restr{\perm(\Gamma;\CR)}\colon \perm(\Gamma;\CR)\op\to \MMod{\CR}.
\]
This defines a functor~$\FP\colon \MMod{\Gamma;\CR}\to \PSh{\CR}{\perm(\Gamma;\CR)}=\CohMack[\CR]{\Gamma}$.
\end{Cons}
\begin{Rem}
\label{Rem:FP-in-cash}%
The notation $\FP$ should evoke `fixed points'. Indeed, when evaluated at one of the additive generators $\CR(\Gamma/H)$ of~$\perm(\Gamma;\CR)$, where $H\le \Gamma$, we get
\[
\big(\FP(M)\big)\big(\CR(\Gamma/H)\big)=\Hom(\CR(\Gamma/H),M)\cong M^{H}.
\]
\end{Rem}
\begin{Rem}
\label{Rem:LP}%
The functor $\FP$ admits a left adjoint $\FPL\colon \CohMack[\CR]{\Gamma}\to \MMod{\Gamma;\CR}$ given by left Kan extension of the inclusion $\perm(\Gamma;\CR)\into\MMod{\Gamma;\CR}$ along the ($\CR$-linear) Yoneda embedding $\perm(\Gamma;\CR)\into\PSh{\CR}{\perm(\Gamma;\CR)}$:
\[
\xymatrix{
  \perm(\Gamma;\CR)
  \ar@{ >->}[r]
  \ar@{ >->}[d]_-{\textrm{Yoneda}}
  &
  \MMod{\Gamma;\CR}
  \\
  \PSh{\CR}{\perm(\Gamma;\CR)}=\CohMack[\CR]{\Gamma}
  \ar@{.>}[ur]_(.6){\FPL}
}
\]
The value of $\FPL$ on a cohomological Mackey functor~$M$ is the colimit $\colim_{P\to M}P$ in $\MMod{\Gamma;\CR}$ of all finitely generated permutation modules $P$ over~$M$ in~$\CohMack[\CR]{\Gamma}$.
This is a standard universal property of the presheaf category.
The functor $\FPL$ is automatically a colimit-preserving tensor functor (\cf \Cref{Rem:coh-mackey-tensor}).
\end{Rem}

\begin{Lem}
\label{Lem:FP-omnibus}%
The functor $\FP:\MMod{\Gamma;\CR}\to\CohMack[\CR]{\Gamma}$ of \Cref{Cons:FP}
\begin{enumerate}[\rm(a)]
\item is $\CR$-linear and lax monoidal,
\item preserves filtered colimits (and thus coproducts), and
\item is fully faithful.
\end{enumerate}
\end{Lem}
\begin{proof}
The first item is obvious (as right adjoint of a tensor functor) and so is the second by \Cref{Rem:FP-in-cash}.
For the last item, we evaluate the counit of the adjunction $\FPL\adj\FP$ on a discrete module $M\in\MMod{\Gamma;\CR}$, and find that it is given by the canonical morphism
\[
\colim_{P\to M}P\to M
\]
where the colimit is indexed by finitely generated permutation modules $P$ over $M$ (\cf \Cref{Rem:LP}).
As $M$ is discrete, this morphism is surjective (\cf \Cref{Rec:discrete-modules}).
To prove that the morphism is injective as well, let $f:P\to M$ be a map from a permutation module $P$, and let $p\in P$ such that $f(p)=0$.
There exists some $H\le \Gamma$ such that $p\in P^H$, thus a map $\CR(\Gamma/H)\to P$ which sends $[e]_H$ to $p$.
In the indexing category for the colimit we therefore find a span
\[
\xymatrix@R=.1em{
  &(P\xto{f} M)\\
  (\CR(\Gamma/H)\xto{0} M)
  \ar[ru]^-p
  \ar[rd]_-0
  \\
  &(0\xto{0} M)
}
\]
from which it follows that $p\in P$ vanishes in the colimit.
\end{proof}

We can now identify \emph{projective} cohomological Mackey functors.
\begin{Prop}[\cf {\cite[Theorem~16.5]{thevenaz-webb:structure-mackey}}]
\label{Prop:perm-proj-comack}%
The fixed-point functor $\FP$ induces equivalences of tensor categories:
\[
\xymatrix{
  \Perm(\Gamma;\CR)^{\natural}
  \ar[r]_-{\sim}^-{\FP}
  &
  \Proj(\CohMack[\CR]{\Gamma})
  \\
  \perm(\Gamma;\CR)^{\natural}
  \ar[r]_-{\sim}^-{\FP}
  \ar@{ >->}[u]
  &
  \proj(\CohMack[\CR]{\Gamma})
  \ar@{ >->}[u]
}
\]
where $\Proj$ stands for the subcategory of projective objects and $\proj$ for that of the finitely presented projective objects, in the Grothendieck category $\CohMack[\CR]{\Gamma}$.
\end{Prop}
\begin{proof}
By Yoneda, we have for every $P\in\perm(\Gamma;\CR)$ that $\Hom_{\CohMack[\CR]{\Gamma}}(P,?)\cong ?(P)$ preserves all colimits (computed objectwise). It follows that every $P$ in~$\perm(\Gamma;\CR)$ is finitely presented projective in~$\CohMack[\CR]{\Gamma}$. Note that on~$\perm(\Gamma;\CR)$, the functor~$\FP$ is nothing but the Yoneda embedding. In particular it is fully-faithful, and remains so on $\perm(\Gamma;\CR)^\natural$ and on $\Perm(\Gamma;\CR)^\natural$ (closure under coproducts). So the two functors~$\FP$ of the statement are fully faithful and take values inside projectives, as indicated. Only their essential surjectivity remains to be seen.

Let $M\colon \perm(\Gamma;\CR)\op\to \MMod{\CR}$ be in~$\CohMack[\CR]{\Gamma}$.
Again by Yoneda, the map
\[
\Big(\coprod_{P\in \perm(\Gamma;\CR)\atop P\to M} \! P \ \Big)\quad\too \quad M
\]
is an epimorphism in~$\CohMack[\CR]{\Gamma}$. If $M$ is projective then this map must admit a section, showing that $M$ is a direct summand of a coproduct of~$P$'s. If moreover $M$ is finitely presented, that section must factor \via a finite coproduct.
\end{proof}

We now come to the announced equivalence~\eqref{eq:perm-cohmack}.
In view of \Cref{Prop:perm-proj-comack}, consider the functor $\FP\colon \K(\Perm(\Gamma;\CR))\to \K(\Proj(\CohMack[\CR]{\Gamma}))$ on chain complexes. Post-composing with the canonical quotient functor and pre-composing with the inclusion of \Cref{Rem:DPerm-switch}, we obtain
\begin{equation*}
\DPerm(\Gamma;\CR)\into \K(\Perm(\Gamma;\CR))\xto{\FP} \K(\Proj(\CohMack[\CR]{\Gamma}))\onto\D(\CohMack[\CR]{\Gamma}).
\end{equation*}
\begin{Cor}
\label{Cor:DPerm-comack}%
The above is an equivalence of tensor triangulated categories:
\begin{equation}
\label{eq:FP-perm-comack}%
\FP:\DPerm(\Gamma;\CR)\isoto\D(\CohMack[\CR]{\Gamma}).
\end{equation}
\end{Cor}
\begin{proof}
The functor~\eqref{eq:FP-perm-comack} is a coproduct-preserving triangulated functor, since it is defined as the composite of three such functors. It restricts to an equivalence
\begin{equation}
\label{eq:FP-perm-comack-compacts}
\Kb(\perm(\Gamma;\CR)^\natural)\isoto\Kb(\proj(\CohMack[\CR]{\Gamma}))
\end{equation}
by \Cref{Prop:perm-proj-comack} and the fact that, as for every abelian category, the functor
\[
\Kb(\proj(\CohMack[\CR]{\Gamma}))\to\D(\CohMack[\CR]{\Gamma})
\]
is fully faithful.
Since the left-hand side of~\eqref{eq:FP-perm-comack-compacts} is the compact part of $\DPerm(\Gamma;\CR)$ by \Cref{Cor:DPerm-compacts}, it suffices to prove that the right-hand side of~\eqref{eq:FP-perm-comack-compacts} is the compact part of $\D(\CohMack[\CR]{\Gamma})$ and generates the latter.
For every $P\in\perm(\Gamma;\CR)$, we have
\begin{align}
\label{eq:hom-proj-D}%
  \Hom_{\D(\CohMack[\CR]{\Gamma})}(\FP(P),X)
  &=\Hom_{\K(\CohMack[\CR]{\Gamma})}(\FP(P),X)\\\notag
  &=\Hm_0(X(P)).
\end{align}
Since homology and evaluation at $P$ commute with coproducts, this shows that the right-hand side of~\eqref{eq:FP-perm-comack-compacts} consists of compact objects in $\D(\CohMack[\CR]{\Gamma})$. Similarly, these $\FP(P)$ generate $\D(\CohMack[\CR]{\Gamma})$ since a complex~$X$ of cohomological Mackey functors that is right-orthogonal to $\FP(P)[i]$ for all~$i\in \bbZ$ must have trivial homology by~\eqref{eq:hom-proj-D}, \ie be zero in the derived category.

It remains to show that the equivalence in the statement is compatible with the tensor structures.
We know that $\FP$ of \Cref{Lem:FP-omnibus} is lax monoidal, hence so is the functor~\eqref{eq:FP-perm-comack}.
Since the equivalence on compacts~\eqref{eq:FP-perm-comack-compacts} is tensor, and since the tensor product commutes with coproducts in each variable, the claim follows.
\end{proof}

\begin{Rem}
\label{Rem:Mackey-algebras}%
In~\cite[\S\,16]{thevenaz-webb:structure-mackey} it was shown that $\CohMack[\CR]{G}$ for $G$ finite is the category of modules over the so-called \emph{cohomological Mackey algebra}. The interested reader can verify that a similar description of $\CohMack[\CR]{\Gamma}$ exists for profinite groups~$\Gamma$, at the cost of using the possibly \emph{non-unital} ring
\[
\cE=\sqcup_{H,K\le \Gamma}\Hom_{\perm(\Gamma;\CR)}(\CR(\Gamma/H),\CR(\Gamma/K)).
\]
The multiplication of this \emph{generalized cohomological Mackey algebra} is given by composition.
Cohomological Mackey functors for~$\Gamma$ are then identified with unital right modules over~$\cE$ (that is, right modules~$M$ such that $M\cdot\cE=M$).

We found it easier to use modules over an additive category, \ie presheaves (as in \Cref{Cor:comack-presh-perm}), rather than modules over non-unital rings.
\end{Rem}
\begin{Rem}
\label{Rem:DPerm-KPerm-recollement}%
Let us return to the localization $\K(\Perm(\Gamma;\CR))\onto\DPerm(\Gamma;\CR)$ of \Cref{Def:DPerm} (\cf \Cref{Rem:DPerm-Bousfield}). First note that $\K(\Perm(\Gamma;\CR))$ is a triangulated category with small coproducts, hence is idempotent-complete. It follows easily that $\K(\Perm(\Gamma;\CR))\cong\K(\Perm(\Gamma;\CR)^\natural)$. So we can identify $\K(\Perm(\Gamma;\CR))$ with the homotopy category of projectives $\K(\Proj(\CohMack[\CR]{\Gamma}))$ by \Cref{Prop:perm-proj-comack}.

Assume now that $\Gamma=G$ is finite and that $\CR$ is coherent. In that case, $\CohMack[\CR]{\Gamma}$ is the category of modules over the cohomological Mackey algebra~$\cE$ of~\cite{thevenaz-webb:structure-mackey}, as in \Cref{Rem:Mackey-algebras}. And the ring~$\cE$ remains coherent. We can then invoke \cite{neeman:K-flat} to conclude that the homotopy category of projectives $\cE$-modules is compactly generated. Hence so is $\K(\Perm(G;\CR))$ in that case.
Applying the general results of \Cref{Rec:Neeman-loc}\,\eqref{it:Neeman-recollement} to the triangulated category $\cT=\K(\Perm(G;\CR))$ and to the set of compact objects~$\cG=\SET{\CR(G/H)}{H\le G}$ we obtain the following left-hand recollement of triangulated categories
\[
\xymatrix@R=2em{
\DPerm(G;\CR)
\ar@{ >->}@/_2pc/[d]_{\incl}
\ar@{ >->}@/^2pc/[d]^{}
\\
\K(\Perm(G;\CR))
\ar@{->>}@/^2pc/[d]^{}
\ar@{->>}@/_2pc/[d]_{}
\ar@{->>}[u]^{Q}
\\
\K_{\Gac}(\Perm(G;\CR))
\ar@{ >->}[u]^{\incl}
}
\qquad\qquad
\xymatrix@R=2em{
\D(\CohMack[\CR]{G})
\ar@{ >->}@/_2pc/[d]_{\incl}
\ar@{ >->}@/^2pc/[d]^{}
\\
\K(\Proj(\CohMack[\CR]{G}))
\ar@{->>}@/^2pc/[d]^{}
\ar@{->>}@/_2pc/[d]_{}
\ar@{->>}[u]^{Q}
\\
\K_{\ac}(\Proj(\CohMack[\CR]{G}))
\ar@{ >->}[u]^{\incl}
}
\]
Under the equivalences of \Cref{Prop:perm-proj-comack} and \Cref{Cor:DPerm-comack}, this recollement can be translated into the recollement depicted on the right-hand side above.

At the moment, we do not know if this can be extended to profinite groups. The main issue is whether $\K(\Proj(\CohMack[\CR]{\Gamma}))$ remains compactly generated.
\end{Rem}

\section{Sheaves with transfers}
\label{sec:Shtr}

In Voevodsky's approach to motives of algebraic varieties over a field, a central concept is the one of \emph{transfers}.
His observation was that for a well-behaved theory it is not enough to consider morphisms of varieties but one has to allow for (finitely) multi-valued maps.
Equivalently, one has to allow certain `wrong-way' morphisms, that is, transfers.
In this section we recall this notion and some basic facts about sheaves with transfers.

The mention of `wrong-way' morphisms and multi-valued maps should ring a bell.
As we saw in \Cref{sec:discrete-perm}, adding `wrong-way' morphisms to $\Gamma$-sets results in the span category and this is only one step away from permutation modules.
The other goal of this section then is to explain the connections between transfers, permutation modules and cohomological Mackey functors.
The basic dictionary between the two sides is provided by Galois theory, and looks as follows (with notation to be introduced in the present section), cf.~\cite{kahn-Yamazaki:DM-Weil-correction}.

\medskip
\begin{center}
\begin{tabular}[c]{ccc}
  \textbf{algebraic geometry}&\hphantom{phantom}&\textbf{representation theory}\\
  \hline{}\\
  $\Sm^0_\FF$&&$\gasets$\\
  $\Span(\Sm^0_\FF)$&&$\Span(\gasets)$\\
  $\Omega(\FF)$&&$\Omega(\Gamma)$\\
  $\Corr_\FF\otimes\CR$&&$\perm(\Gamma;\CR)$\\
  $\Sh{\Nis}{\Corr_\FF}{\CR}$&&$\CohMack[\CR]{\Gamma}$
\end{tabular}
\end{center}
\medskip

\begin{Hyp}
\label{Hyp:field-galois-group}%
Let $\FF$ be a field with a choice of separable algebraic closure $\FFsep$ and absolute Galois group~$\Gamma=\Gamma_{\FF}:=\Gal(\FFsep/\FF)$.
All schemes are assumed to be separated and of finite type over their base field (which is often $\FF$, or a finite extension thereof).
\end{Hyp}

\begin{Rec}
  \label{correspondences}%
    Let $X$ and $Y$ be smooth $\FF$-schemes.
    The free abelian group on integral subschemes of $X\times Y$, finite and surjective over a connected component of~$X$, is denoted $\corr(X,Y)$, or $\corr_{\FF}(X,Y)$ to be more precise.
    These are the \emph{finite correspondences} from $X$ to~$Y$.
    The category whose objects are smooth $\FF$-schemes and whose morphisms are finite correspondences is denoted~$\Corr_{\FF}$.
    The composition is defined in terms of push-forward and intersection of cycles~\cite[\S\,9.1]{cisinski-deglise:DM}.
    There is a faithful functor $\Sm_{\FF}\to\Corr_{\FF}$ from smooth $\FF$-schemes to the category of $\FF$-correspondences, which takes a morphism to its graph (and is the identity on objects).
    Note that the cartesian product endows $\Corr_{\FF}$ with the structure of a tensor category (finite biproducts are given by the disjoint union of schemes).
\end{Rec}

\begin{Exa}
  \label{0-corr}%
  Let $X$ be a zero-dimensional smooth $\FF$-scheme, \ie a scheme \'etale over $\Spec(\FF)$.
  Then $X=\coprod_{i=1}^n\Spec(\KK_i)$ is a finite disjoint union of spectra of finite separable extensions of $\FF$.
  A finite correspondence from $X$ to $Y$ is a finite linear combination of closed points of $X\times Y=\coprod_iY_{\KK_i}$, that is, $\corr(X,Y)$ is the group of zero cycles in $X\times Y$, denoted $Z_0(X\times Y)$.
  Moreover, the finite correspondences in $\corr(Y,X)$ are finite linear combinations of connected components of $Y\times X$, that is, the free abelian group on~$\pi_0(Y\times X)$.

  We denote the full subcategory of $\Corr_\FF$ spanned by the zero-dimensional smooth $\FF$-schemes by~$\Corr^0_{\FF}$.
  Of course, it is a tensor subcategory of~$\Corr_{\FF}$.
\end{Exa}

\begin{Rec}
\label{Rec:fet-g-sets}%
Let $X$ be a scheme over $\FF$ and consider the set of $\FFsep$-points:
\[
X(\FFsep):=\Hom_{\FFsep}(\Spec(\FFsep),X_{\FFsep}).
\]
The action of $\Gamma$ on $X_{\FFsep}$ endows $X(\FFsep)$ with the structure of a $\Gamma$-set, and elaborating a bit one finds a functor
\[
\Schm_{\FF}\to\Gasets.
\]
Galois theory (for example in the form of~\cite[Expos\'{e}~V]{sga1}) tells us that this functor restricts to an equivalence
\[
\Sm_{\FF}^0\isoto\gasets
\]
between finite \'etale $\FF$-schemes and finite $\Gamma$-sets.
The inverse of this equivalence sends a transitive $\Gamma$-set $\Gamma/H$ to $\Spec(\FFsep^H)$.
\end{Rec}

\begin{Rem}
\label{Rem:corr-spans}%
From the equivalence of \Cref{Rec:fet-g-sets} we deduce an equivalence of span categories:
\[
\Span(\Sm_{\FF}^0)\isoto\Span(\gasets).
\]
Let us denote by $\Omega(\FF)$ the additive category obtained from $\Span(\Sm_\FF^0)$ by group completing the monoids of homomorphisms.
Thus a further equivalence (\cf \Cref{Def:profinite-Mackey-functor}):
\[
\Omega(\FF)\isoto\Omega(\Gamma).
\]
The objects of $\Omega(\FF)$ and $\Corr^0_\FF$ coincide, and are in both cases the zero-dimensional smooth $\FF$-schemes.
Morphisms between two zero-dimensional smooth $\FF$-schemes $X$ and $Y$ in the two categories are also closely related.
Indeed, in $\Omega(\FF)$ the group of morphisms $X\to Y$ is generated by spans $f:Z\to X\times Y$ with $Z$ another smooth zero-dimensional $\FF$-scheme.
In $\Corr^0_{\FF}$ the group of morphisms $X\to Y$ is generated by spans $f:Z\to X\times Y$ as before, where $f$ is a closed immersion.
(And composition of morphisms coincides.)
Reminding ourselves of \Cref{Rem:span-factorization}, we should view the passage from $\Omega(\FF)$ to $\Corr^0_\FF$ as being analogous to the passage from $\Omega_\CR(\Gamma)$ to $\perm(\Gamma;\CR)$.
And this is indeed precisely right, as one can easily show.
We will deduce it later (\Cref{Corr-perm-equivalence}) from more general considerations but, as an illustration of the notions just introduced, we sketch the main idea here.

Let $f:Z\to X\times Y$ be a span describing a morphism $X\tobar Y$ in $\Omega(\FF)$ and assume that $Z$ is connected.
We may push-forward cycles along the finite map $f$,
\[
f_*[Z]\in Z_0(X\times Y)=\corr(X,Y),
\]
as seen in \Cref{0-corr}.
Recall that this is very concrete: $f$ factors through one of the points $P$ of $X\times Y$, and is described simply by a finite field extension $\FF(P)\subseteq \FF(Z)$.
The $0$-cycle $f_*[Z]$ is then nothing but $[\FF(Z):\FF(P)]\cdot P$.

In particular, let $\FF\subseteq\KK\subseteq\LL$ be a pair of finite separable field extensions and consider the span
\begin{equation}
\label{eq:special-etale-span}%
\vcenter{\xymatrix@R=.5em{
  &
  \Spec(\LL)
  \ar[ld]_{\pi}
  \ar[rd]^{\pi}
  \\
  \Spec(\KK)
  &&
  \Spec(\KK)
}}
\end{equation}
corresponding to the composite morphism $\KK\otimes_\FF\KK\into\LL\otimes_\FF\LL\xonto{m} \LL$ with $m$ being the multiplication in $\LL$.
This also factors as $\KK\otimes_\FF\KK\xonto{m}\KK\into\LL$ from which we deduce that the closed point $P$ in this case is $\Spec(\KK)$ and its $0$-cycle in $\Spec(\KK)\times\Spec(\KK)$ is the graph of the identity morphism.
It follows that the image of~\eqref{eq:special-etale-span} in correspondences is $[\LL:\KK]\cdot\Id_{\Spec(\KK)}$. So $\Omega(\FF)\to\Corr_\FF$ factors through the analogous quotient as in \Cref{Omega-perm-quotient}.
It is then not difficult to see that the induced functor on the quotient is an equivalence.
\end{Rem}

\begin{Rec}[{\cite[\S\,10]{cisinski-deglise:DM}}]
\label{sheaves-with-transfers}%
A \emph{presheaf with transfers} is an additive presheaf on $\Corr_{\FF}$ with values in $\CR$-modules.
Such a presheaf is called a \emph{sheaf with transfers} if its restriction to $\Sm_{\FF}$ is a sheaf. Here we are interested in the Nisnevich and the \'etale topology.
Recall that covering families in the latter are finite families $(X_i\to X)$ of \'etale morphisms that are jointly surjective.
For the Nisnevich topology one requires in addition that for every point $x\in X$ there exists $i$ and $y\in X_i$ mapping to $x$ and inducing an isomorphism $\kappa(x)\isoto\kappa(y)$ on residue fields.

This gives rise to Grothendieck abelian categories $\Sh{\Nis}{\Corr_{\FF}}{\CR}$ and $\Sh{\et}{\Corr_{\FF}}{\CR}$.
For every smooth $\FF$-scheme $X$, the associated presheaf with transfers $\corr_\FF(-,X)\otimes\CR$ is a sheaf for both topologies, and these objects form a dense generating family for the categories of sheaves with transfers.
We employ similar terminology for presheaves on the full subcategory~$\Corr^0_{\FF}$, yielding Grothendieck abelian categories $\Sh{\Nis}{\Corr^0_{\FF}}{\CR}$ and $\Sh{\et}{\Corr^0_\FF}{\CR}$.
\end{Rec}

\begin{Rec}[{\cite[\S\,10.3]{cisinski-deglise:DM}}]
\label{Rec:sheafification-transfers}%
In the sequel we will write $\tau$ for any of the two topologies $\Nis$ or $\et$.
We will write invariably $\otr$ for the functor on (pre)sheaves which `forgets transfers', \ie is induced by restriction along $\Sm_\FF^{(0)}\into\Corr_\FF^{(0)}$.
The functor $\otr$ is faithful and exact (in fact, it commutes with all limits and colimits) hence it is conservative.

The canonical inclusion $\Sh{\tau}{\Corr^{(0)}_\FF}{\CR}\into\PSh{\oplus}{\Corr^{(0)}_\FF;\CR}$ into the category of additive presheaves admits a left adjoint $\atau$ such that the canonical natural transformation is an equivalence $\atau\otr\isoto\otr\atau$.
In other words, the sheafification of the underlying presheaf without transfers admits a canonical structure of presheaf with transfers.
In particular, the sheafification functor at the level of (pre)sheaves with transfers is exact as well.

Finally, sheafification and Day convolution endow the categories of sheaves with transfers with a closed tensor structure extending the one on~$\Corr_{\FF}^{(0)}$.
\end{Rec}

We will use the following technical result.
\begin{Lem}
  \label{sheaves-with-transfers-properties}%
  The inclusion $\iota:\Corr^0_{\FF}\into\Corr_{\FF}$ and \'etale sheafification induce a commutative square of left adjoint exact tensor functors
  \[
    \xymatrix{
      \Sh{\Nis}{\Corr^0_{\FF}}{\CR}
      \ar@{ >->}[r]^{\iota_!}
      \ar[d]^{\aet}
      &
      \ \Sh{\Nis}{\Corr_{\FF}}{\CR}
      \ar[d]^{\aet}
      \\
      \Sh{\et}{\Corr^0_{\FF}}{\CR}
      \ar@{ >->}[r]^{\iota_!}
      &
      \ \Sh{\et}{\Corr_{\FF}}{\CR}
    }
  \]
  where both horizontal arrows $\iota_!$ are fully faithful.
\end{Lem}
\begin{proof}
  The inclusion $\iota:\Corr_\FF^0\into\Corr_\FF$ induces an adjunction $\hat{\iota}_!\dashv\hat{\iota}^*=(-)\circ\iota$ at the level of presheaves with values in $\CR$-modules, which restricts to an adjunction (denoted by the same symbols) at the level of additive presheaves.
  It is clear that the restriction $\hat{\iota}^*=(-)\circ\iota$ preserves $\tau$-sheaves with transfers.
  We denote the induced functor by $\iota^*:\Sh{\tau}{\Corr_\FF}{\CR}\to\Sh{\tau}{\Corr_\FF^0}{\CR}$.
  It admits a left adjoint $\iota_!=\atau\hat{\iota}_!$,
  and consequently the square of left adjoints in the statement of the lemma commutes, as claimed.
  It remains to prove that the horizontal arrows are fully faithful and exact tensor functors.

  Since the tensor structure on sheaves with transfers is obtained from Day convolution and sheafification from the tensor structure on $\Corr_\FF^{(0)}$, and since the inclusion $\iota:\Corr_\FF^0\into\Corr_\FF$ is tensor, it follows immediately that $\iota_!$ is tensor as well.

  For full-faithfulness and exactness we start with the following observation.
  There is a canonical isomorphism $\otr\iota^*\cong\iota^*\otr$ where on the right hand side, $\iota^*$ denotes the analogous restriction functor on sheaves without transfer.
  We claim that the induced comparison morphism $\iota_!\otr\to\otr\iota_!:\Sh{\tau}{\Corr^0_\FF}{\CR}\to\Sh{\tau}{\Sm_\FF}{\CR}$ is an isomorphism too.
  Since all functors preserve colimits, it suffices to show that the comparison morphism is invertible when evaluated on `representable' sheaves with transfers $\corr^0_\FF(-,X)\otimes\CR$, where $X$ is an \'etale $\FF$-scheme.
  This follows from~\cite[Corollary~2.1.9]{cisinski-deglise:etale-motives} (in fact, both sides are equal to the $\CR$-linear $\tau$-sheaf on $\Sm_\FF$ represented by $X$).
  We conclude that exactness of $\iota_!$ at the level of sheaves with transfers would follow from the same property of $\iota_!$ at the level of sheaves without transfers (since $\otr$ is faithful exact).
  And since fully faithfulness of $\iota_!$ is equivalent to the unit $\Id\to\iota^*\iota_!$ being an equivalence, this would also follow from the same property of $\iota_!$ at the level of sheaves without transfers.

  The inclusion $\iota:\Sm_{\FF}^0\into\Sm_{\FF}$ is a continuous and cocontinuous functor for both topologies~\cite[III, Corollary~3.4]{sga4.1} and it follows from general topos theory~\cite[III, Proposition~2.6]{sga4.1} that
  \begin{equation}
    \iota_!:\Sh{\tau}{\Sm^0_{\FF}}{\CR}\to\Sh{\tau}{\Sm_{\FF}}{\CR}\label{eq:iota-wo-transfers}
  \end{equation}
  is fully faithful.
  The inclusion $\iota:\Sm_{\FF}^0\into\Sm_{\FF}$ also admits a left adjoint which sends a connected smooth $\FF$-scheme $X$ to the spectrum of the separable closure of $\FF$ in $\Hm^0(X,\mathcal{O}_X)$.
  It follows from~\cite[III, Proposition~2.5]{sga4.1} that~(\ref{eq:iota-wo-transfers}) is exact, and this concludes the proof.
\end{proof}

\begin{Not}
  \label{fiber-functor}%
  Let $M\in\Sh{\et}{\Corr_{\FF}^0}{\CR}$ be an \'etale sheaf with transfers.
  We define the $\CR$-module $M(\FFsep)$ as the following colimit in $\CR$-modules
  \[
    M(\FFsep):=\colim_{\KK}M(\Spec(\KK)),
  \]
  where $\KK$ runs over the finite field extensions of $\FF$ contained in $\FFsep$.
  This $\CR$-module $M(\FFsep)$ comes with a canonical action of~$\Gamma=\Gal(\FFsep/\FF)$, inducing a functor
  \begin{equation}
    \label{eq:fiber-functor}%
    \Psi_{\FFsep}:\Sh{\et}{\Corr_{\FF}^0}{\CR}\to\MMod{\Gamma;\CR}.
  \end{equation}
\end{Not}

\begin{Lem}
  \label{Galois-equivalence}%
  The functor $\Psi_{\FFsep}$ of~(\ref{eq:fiber-functor}) is an equivalence of $\CR$-linear tensor categories.
\end{Lem}
\begin{proof}
  It is clear that $M(\FFsep)$ only depends on the sheaf without transfers underlying~$M$, that is, we have a factorization
  \[
    \Psi_{\FFsep}:\Sh{\et}{\Corr_{\FF}^0}{\CR}\xto{\otr}\Sh{\et}{\Sm_{\FF}^0}{\CR}\isoto\MMod{\Gamma;\CR}.
  \]
  The second functor is the well-known $\CR$-linear tensor equivalence induced by Galois theory~\cite[VIII, Corollaire~2.2]{sga4.2}, and it therefore suffices to show that $\otr$ is an equivalence of $\CR$-linear tensor categories too.
  It admits a left adjoint $\atr$ for formal reasons, and we want to show that the unit $\eta$ of this adjunction is invertible.
  Since both functors preserve colimits, it suffices to show that $\eta$ is invertible when evaluated on representable sheaves.
  This follows from~\cite[Corollary~2.1.9]{cisinski-deglise:etale-motives}.

  Now, $\atr$ is colimit preserving and fully faithful, and its essential image contains the generating family of $\Sh{\et}{\Corr^0_\FF}{\CR}$ given by the `representable' sheaves with transfers.
  This shows that $\atr$ is an equivalence, with quasi-inverse $\otr$.
  Finally, we note that $\atr$ is $\CR$-linear and tensor, and this completes the proof.
\end{proof}

\begin{Prop}
  \label{Corr-perm-equivalence}%
  Consider the exact tensor functor $\Psi_{\FFsep}\circ\aet:\Sh{\Nis}{\Corr^0_{\FF}}{\CR}\to\MMod{\Gamma;\CR}$.
  It restricts to an equivalence of $\CR$-linear tensor categories
  \[
    \Corr^0_{\FF}\otimes\,\CR\isoto \perm(\Gamma;\CR).
  \]
\end{Prop}
\begin{proof}
  The composite
  \[
    \Corr^0_{\FF}\otimes\CR\to\Sh{\Nis}{\Corr^0_{\FF}}{\CR}\xto{\aet}\Sh{\et}{\Corr^0_{\FF}}{\CR}
  \]
  is fully faithful, and it is easy to see that under the equivalence $\Psi_{\FFsep}$ the image corresponds precisely to the permutation modules.
\end{proof}

\begin{Rem}
  \label{sheaves-picture}%
  We may summarize \Cref{Corr-perm-equivalence} by the commutative diagram
\begin{equation*}
  \vcenter{\xymatrix@R=1.5em{
    \perm(\Gamma;\CR)
    \ar@{ >->}[dd]
    &
    \Corr_{\FF}^0\otimes\CR
    \ar[l]_-{\sim}
    \ar@{ >->}[d]
    \\
    &
    \Sh{\Nis}{\Corr_{\FF}^0}{\CR}
    \ar[d]^{\aet}
    \\
    \MMod{\Gamma;\CR}
    &
    \Sh{\et}{\Corr_{\FF}^0}{\CR}
    \ar[l]^{\sim}_{\Psi_{\FFsep}}
    }}\label{eq:yoneda-picture}
  \end{equation*}
  where the left vertical arrow is the canonical inclusion, and the first vertical arrow on the right is the Yoneda embedding.
\end{Rem}

\begin{Rem}
\label{Rem:sheaves-mackey-functors}%
The Nisnevich topology on $\Corr^0_{\FF}$ is very simple:
Every cover splits, and a presheaf with transfers on $\Corr^0_{\FF}$ (\ie an additive contravariant functor) is therefore automatically a Nisnevich sheaf with transfers.
It follows that we may extend $\Psi_{\FFsep}$ to Nisnevich sheaves with transfers and deduce the following consequence.
\end{Rem}

\begin{Cor}
\label{sheaves-transfers-mackey}%
There is an equivalence of $\CR$-linear (abelian) categories
\[
\Psi_{\FFsep}:\Sh{\Nis}{\Corr^0_\FF}{\CR}\isoto\CohMack[\CR]{\Gamma}
\]
between sheaves with transfers on zero-dimensional smooth schemes, and cohomological Mackey functors.
\qed
\end{Cor}

\begin{Rem}
\label{FP}%
The equivalences established in the last few results are related through the two commutative squares (of left and right adjoints, respectively):
\[
\xymatrix{
  \CohMack[\CR]{\Gamma}
  \ar@{}[d]|\dashv
  \ar@/^-1pc/[d]_\FPL
  &
  \Sh{\Nis}{\Corr^0_\FF}{\CR}
  \ar@{}[d]|\dashv
  \ar@/^-1pc/[d]_{\aet}
  \ar[l]_{\Psi_{\FFsep}}^\sim
  \\
  \MMod{\Gamma;\CR}
  \ar@/^-1pc/[u]_{\FP}
  &
  \Sh{\et}{\Corr^0_\FF}{\CR}
  \ar[l]^\sim_{\Psi_{\FFsep}}
  \ar@/^-1pc/[u]
}
\]
Here, the right adjoint to \'etale sheafification is just the obvious inclusion, and $\FPL\dashv\FP$ is the adjunction of \Cref{Rem:LP}.
\end{Rem}

\section{Artin motives}
\label{sec:motivic}

In this section, we want to explain the connection with Artin motives alluded to in the introduction.
This is originally due to Voevodsky~\cite[\S\,3.4]{Voevodsky00}.
We keep the notation and hypotheses of \Cref{Hyp:field-galois-group}, that is, $\FF$ is a field with absolute Galois group $\Gamma=\Gal(\FFsep/\FF)$.

To avoid cumbersome notation, we will from now on write $X\in\Sh{\tau}{\Corr_\FF}{\CR}$ for the sheaf with transfers $\corr_\FF(-,X)\otimes\CR$ whenever $X$ is a smooth $\FF$-scheme.
\begin{Rec}
  \label{DMeff}%
  The category of \emph{effective motives} over $\FF$ (with $\CR$-linear coefficients) is the Verdier localization of the derived category of sheaves with transfers
  \begin{equation}
  \label{eq:DMeff}
    \DMbigeff(\FF;\CR):=\frac{\D(\Sh{\Nis}{\Corr_{\FF}}{\CR})}{\Loc{\SET{\cone(\AA^1_X\to X)}{X\in\Sm_{\FF}}}}
  \end{equation}
obtained by inverting $\AA^1_X\to X$ for every smooth $\FF$-scheme~$X$.

The compact part of $\DMbigeff(\FF;\CR)$ is called the category of \emph{effective geometric motives}, and is denoted
\[
\DMeff(\FF;\CR):=\DMbigeff(\FF;\CR)^c.
\]

Let $X$ be a smooth $\FF$-scheme.
The image of $X$ in $\DMeff(\FF;\CR)$ is called the \emph{(effective) motive} of $X$. (Since we will deal exclusively with effective motives in the sequel, we will often drop the adjective.)
\end{Rec}

\begin{Rem}
  \label{DMeff-structure}%
  The tensor product on sheaves with transfers induces a tensor product on the derived category~\cite[\S\,5, or 11.1.2]{cisinski-deglise:DM}.
  The kernel of the Verdier localization in~(\ref{eq:DMeff}) is an ideal hence $\DMbigeff(\FF;\CR)$ inherits the structure of a tensor category.
  Also, the triangulated category $\D(\Sh{\Nis}{\Corr_{\FF}}{\CR})$ is compactly generated by (the sheaves with transfers representing) smooth $\FF$-schemes. (This follows from the finite cohomological dimension with respect to the Nisnevich topology~\cite[1.2.5]{kato-saito:global-CF-arithmetic}.)
  It follows that $\DMbigeff(\FF;\CR)$ is compactly generated by motives of smooth schemes.
\end{Rem}

\begin{Prop}
\label{Prop:derived-embedding-artin}%
\begin{enumerate}[\rm(a)]
\item
The functor induced by \Cref{sheaves-with-transfers-properties},
\[
\iota_!:\D(\Sh{\Nis}{\Corr^0_{\FF}}{\CR})\into\D(\Sh{\Nis}{\Corr_{\FF}}{\CR}),
\]
is tensor triangulated and fully faithful.
\item
Its composite with the quotient of~\eqref{eq:DMeff} remains fully faithful:
\[
\iota_!:\D(\Sh{\Nis}{\Corr^0_{\FF}}{\CR})\into\DMbigeff(\FF;\CR).
\]
\end{enumerate}
\end{Prop}
\begin{proof}
The functor $\iota_!$ of \Cref{sheaves-with-transfers-properties} was shown to be tensor, fully faithful and exact.
As it also has an exact right adjoint $\iota^*$, it follows that the unit of the adjunction at the level of derived categories remains an isomorphism.
This shows the first statement.

For the second statement, it suffices to prove that the image of each $0$-dimensional smooth $\FF$-scheme in $\D(\Sh{\Nis}{\Corr_{\FF}}{\CR})$ is local with respect to the morphisms
  \[
    \AA^1_X\to X
  \]
  for each smooth $\FF$-scheme $X$.
  Thus let $\FF'/\FF$ be a finite separable field extension.
  We want to prove that the map
  \[
    \Hom_{\D(\Sh{\Nis}{\Corr_{\FF}}{\CR})}(X,\Sigma^i\Spec(\FF'))\to\Hom_{\D(\Sh{\Nis}{\Corr_{\FF}}{\CR})}(\AA^1_X,\Sigma^i\Spec(\FF'))
  \]
  is bijective for each $i\in\bbZ$.
  As the objects of $\Corr_\FF^0$ are their own tensor duals, this is equivalent to showing
  \[
    \Hom_{\D(\Sh{\Nis}{\Corr_{\FF}}{\CR})}(X',\Sigma^i\CR)\to\Hom_{\D(\Sh{\Nis}{\Corr_{\FF}}{\CR})}(\AA^1_{X'},\Sigma^i\CR)
  \]
  bijective, where we set $X'=X_{\FF'}$.
  But by \cite[Proposition~3.1.8]{Voevodsky00}, this identifies with the canonical map
  \[
    \Hm_{\Nis}^i(X',\CR)\to\Hm_{\Nis}^i(\AA^1_{X'},\CR).
  \]
  Both sides vanish for $i>0$ by \Cref{nis-cohomology-constant} below and for $i<0$, and the map is an isomorphism for $i=0$:
  \[
    R^{\pi_0(X')}\isoto R^{\pi_0(\AA^1_{X'})}.\qedhere
  \]
\end{proof}

\begin{Lem}
  \label{nis-cohomology-constant}%
  Let $X$ be an irreducible, geometrically unibranch (\eg normal) scheme, and $F$ a constant sheaf of $\CR$-modules on the small Nisnevich site $X_{\Nis}$.
  Then $\Hm^i_{\Nis}(X,F)=0$ for all $i>0$.
\end{Lem}
\begin{proof}
  Let $K(X)$ denote the function field of $X$, and consider the inclusion
  \[
    \eta:\Spec(K(X))\into X
  \]
  which induces a morphism of sites
  \[
    (-)_\eta:X_{\Nis}\to \Spec(K(X))_{\Nis}
  \]
  and associated morphism of topoi
  \[
    \eta^*:\Sh{\Nis}{\Et_X}{\CR}\rightleftarrows\Sh{\Nis}{\Et_{\Spec(K(X))}}{\CR}:\eta_*.
  \]
  It follows directly from the properties of \'etale morphisms with target $X$ in~\cite[Proposition~18.10.7]{EGAIV.4}
  that the canonical morphism
  \begin{equation}
    F\to \eta_*\eta^*F\label{eq:constant-sheaf-iso}
  \end{equation}
  is an isomorphism.
  Note that every Nisnevich cover in $\Et_{\Spec(K(X))}$ splits and hence every sheaf is flabby.
  As direct images preserve flabby sheaves, it follows from the isomorphism~(\ref{eq:constant-sheaf-iso}) that $F$ is flabby as well.
  This concludes the proof.
\end{proof}

\begin{Rem}
  \label{Kbperm-DM-alternative}%
  Alternatively, if $\FF$ is a perfect field, the subcategory of $\AA^1$-local objects in $\D(\Sh{\Nis}{\Corr_{\FF}}{\CR})$ identifies with those complexes whose homology sheaves are $\AA^1$-invariant.
  (In~\cite[Proposition~3.2.3]{Voevodsky00} this is stated only for right-bounded complexes; the general case is in~\cite[Theorem~4.4]{beilinson-vologodsky:dg-voevodsky-motives}.)
  By \Cref{0-corr}, we see that an object of $\Corr^0_{\FF}$ defines an $\AA^1$-invariant Nisnevich sheaf with transfers, and this yields a shorter proof of \Cref{Prop:derived-embedding-artin} for $\FF$ perfect.
\end{Rem}

\begin{Not}
  \label{DAM}%
  We define $\DAMbig(\FF;\CR)$ as the localizing subcategory of $\DMbigeff(\FF;\CR)$ generated by the motives of 0-dimensional smooth $\FF$-schemes, and we call it the category of \emph{Artin motives}.
  It is a compactly generated tensor triangulated category.
  Its compact part, the category of \emph{geometric Artin motives},
  \[
    \DAM(\FF;\CR):=\DAMbig(\FF;\CR)^c
  \]
  can also be described as the thick subcategory of $\DMeff(\FF;\CR)$ generated by the motives of 0-dimensional smooth $\FF$-schemes (\Cref{Rec:Neeman-loc}\,\eqref{it:Neeman-compacts}).
\end{Not}

\begin{Rem}
  \label{DAM-terminology}%
  Recall that $\DMeff(\FF;\CR)$ is a full subcategory of the category of Voevodsky motives $\DM(\FF;\CR)$~\cite[Theorem~4.3.1]{Voevodsky00}.
  The latter is obtained from the former by tensor-inverting the Tate object $R(1)$, and thereby turning each object rigid (for the tensor structure).
  As we observed above, the motives of 0-dimensional smooth $\FF$-schemes are rigid; in fact, they are their own tensor duals.
  It follows that $\DAM(\FF;\CR)$ is already a rigid tensor triangulated category, and this explains why one does not distinguish between an effective and non-effective version of Artin motives.

  Artin representations are typically understood as finite dimensional representations of the absolute Galois group of a field.
  Originally the vector spaces were over the field of complex number, and the base field was the field of rational numbers.
  Later, other base fields were considered, and in the motivic community it is not unusual to consider more general coefficients.
  This sheds some light on the terminology introduced in \Cref{DAM}.
\end{Rem}

The following result completes the picture~\eqref{eq:triangle-derived} discussed in the introduction.
\begin{Cor}
  \label{DPerm-DAM}%
  There are canonical equivalences of tensor triangulated categories
  \[
    \DPerm(\Gamma;\CR)\xisoto{\Psi_{\FFsep}\inv\circ\FP}\D(\Sh{\Nis}{\Corr^0_\FF}{\CR})\xisoto{\iota_!}\DAMbig(\FF;\CR).
  \]
  They restrict to equivalences of tensor triangulated categories of compacts
  \[
    \Kb(\perm(\Gamma;\CR)^{\natural})\isoto \Kb((\Corr_\FF^0\otimes\CR)^{\natural})\isoto \DAM(\FF;\CR).
  \]
\end{Cor}

\begin{proof}
The first equivalence is \Cref{sheaves-transfers-mackey} and \Cref{Cor:DPerm-comack}.
And the fully-faithful functor of \Cref{Prop:derived-embedding-artin} has image~$\DAMbig(\FF;\CR)$.
\end{proof}

\begin{Rem}
  \label{FP-inverse}%
  Of course, the inverse of the composite equivalence sends the motive of a finite separable field extension $\FF\subseteq\FF'$ corresponding to an open subgroup $\Gamma'\le\Gamma$ to the permutation module $\CR(\Gamma/\Gamma')$, \cf \Cref{Rec:fet-g-sets}.
\end{Rem}

\begin{Rem}
  \label{comparison-Voevodsky}%
  Assume $\FF$ is perfect.
  We may identify $\DMbigeff(\FF;\CR)$ with the full subcategory of $\D(\Sh{\Nis}{\Corr_{\FF}}{\CR})$ of complexes with $\AA^1$-invariant homology sheaves (\Cref{Kbperm-DM-alternative}).
  Restrict attention to the full subcategory
  \[
    \DMminuseff(\FF;\CR)
  \]
  spanned by those complexes whose homology is right-bounded (in addition to being $\AA^1$-invariant).
  This is Voevodsky's `big' category of effective motives~\cite{Voevodsky00}.

Note that $\D^{-}(\Sh{\Nis}{\Corr^0_{\FF}}{\CR})$ does not have all coproducts. Yet, it is the smallest triangulated subcategory (of itself)
  containing $\Kb(\perm(\Gamma;\CR))$ and closed under coproducts. So it follows from \Cref{Prop:derived-embedding-artin} that we have an equivalence
  \[
    \D^{-}(\Sh{\Nis}{\Corr^0_{\FF}}{\CR})\isoto\DAMbig(\FF;\CR)\cap\DMminuseff(\FF;\CR).
  \]
  If $\CR=\bbZ$, this recovers~\cite[Proposition~3.4.1]{Voevodsky00}.
\end{Rem}

\begin{Rem}
  \label{etale-realization}%
  Let us follow up on the picture presented in \Cref{sheaves-picture}.
  Replacing Nisnevich by \'etale sheaves in \Cref{DMeff} one obtains similarly a tensor triangulated category $\DMetbigeff(\FF;\CR)$.
  The sheafification functor $\aet$ passes to the level of motives and we obtain a diagram:
\begin{equation}
\label{eq:etale-realization}%
  \vcenter{\xymatrix{
  \DPerm(\Gamma;\CR)
  \ar[r]_-{\sim}^-{\Psi_{\FFsep}\inv\circ\FP}
  \ar[d]
  &
  \D(\Sh{\Nis}{\Corr^0_{\FF}}{\CR})
  \ar@{>->}[r]^-{\iota_!}
  \ar[d]^{\aet}
  &
  \DMbigeff(\FF;\CR)
  \ar[d]^{\aet}
  \\
  \D(\MMod{\Gamma;\CR})
  &
  \D(\Sh{\et}{\Corr^0_{\FF}}{\CR})
  \ar[r]^-{\iota_{!}}
  \ar[l]^{\sim}_{\Psi_{\FFsep}}
  &
  \DMetbigeff(\FF;\CR).
  }}
\end{equation}
  Here the left square is the derived analogue of the commutative square in \Cref{sheaves-picture}, and the right square is (induced from) a derived analogue of the commutative square in \Cref{sheaves-with-transfers-properties}.
\end{Rem}

\begin{Cor}
Assume that $R$ is $n$-torsion with $n\in\bbZ$ invertible in~$\FF$.
Under the equivalence of \Cref{DPerm-DAM}, the \'etale realization on Artin motives corresponds to the canonical functor which is the identity on objects:
\begin{equation}
\label{eq:Re-et-Artin}%
\vcenter{\xymatrix@C=4em{
  \DPerm(\Gamma;\CR)
  \ar[r]^{\textup{Cor.~\ref{DPerm-DAM}}}_{\sim}
  \ar[d]
  &
  \DAMbig(\FF;\CR)
  \ar[d]^{\Reet}
  \\
  \D(\MMod{\Gamma;\CR})
  \ar[r]^{=}
  &
  \D(\MMod{\Gamma;\CR}).
  }}
\end{equation}
\end{Cor}
\begin{proof}
The Rigidity Theorem (in the form of~\cite[Theorem~4.5.2]{cisinski-deglise:etale-motives}) implies that the bottom right horizontal arrow $\iota_!$ in~\eqref{eq:etale-realization} becomes an equivalence.
Recall also that the \'etale realization may then be described as the composite
\begin{equation*}
\label{eq:Re-et}
\Reet:\DMbigeff(\FF;\CR)\xrightarrow{\aet}\DMetbigeff(\FF;\CR)\simeq\D(\MMod{\Gamma;\CR})
\end{equation*}
from the top right to the bottom left in~\eqref{eq:etale-realization}.
Hence~\eqref{eq:Re-et-Artin} commutes.
\end{proof}

\section{Mackey functoriality}
\label{sec:mackey-functoriality}

In previous sections we have established the equivalences between three theories as described in \Cref{fig:main-triangle} and more precisely in~\eqref{eq:triangle-derived}.
Each of these theories comes with a dependence on a profinite group (or base field), and the variance of the theory as a function of the group may be seen to satisfy axioms reminiscent of Mackey functors.
Although this can be formalized as a Mackey (2-)functoriality in the sense of~\cite{balmer-dellambrogio:two-mackey}, we restrict ourselves to the key ingredients without theoretical elaboration.
Namely, we will exhibit the functoriality for each of the theories, and show that it is compatible with the equivalences between the theories.

\begin{Rem}
\label{Rem:mackey-functoriality-Dperm}%
We recalled the restriction, induction, and conjugation functors on discrete modules in \Cref{Rec:Mackey-operations-discrete} and we noted that they restrict to permutation modules.
As each of these operations is both a left and a right adjoint, they pass to $\Kb(\perm(-;\CR))$ and then to $\DPerm(-;\CR)$.
\end{Rem}

\begin{Rem}
\label{Rem:mackey-functoriality-cohmack}%
In order to quickly define the Mackey structure on cohomological Mackey functors, we view them as ($\CR$-linear) presheaves on permutation modules (\Cref{coh-mackey-perm}).
A similar approach is taken by Th\'{e}venaz and Webb~\cite{thevenaz:simple-mackey-functors}, and is equivalent to the more elementary constructions by Yoshida~\cite{Sasaki:green-wielandt}.

Given an open subgroup $\Gamma'\le \Gamma$, an element $\gamma\in \Gamma$, and cohomological Mackey functors $M\in\CohMack[\CR](\Gamma')$, $N\in\CohMack[\CR]{\Gamma}$, we define
\begin{equation}
\label{eq:mackey-functoriality-cohmack}
\rho^{\Gamma}_{\Gamma'}(N)=N\circ\Ind^{\Gamma}_{\Gamma'},\qquad \tau^{\Gamma}_{\Gamma'}(M)=M\circ\Res^{\Gamma}_{\Gamma'},\qquad \sigma_\gamma(M)=M\circ c_{\gamma\inv},
\end{equation}
where we used the corresponding Mackey functoriality on permutation modules (\Cref{Rec:Mackey-operations-discrete}).
As the latter functors are all $\CR$-linear, we conclude that the newly defined functors in~(\ref{eq:mackey-functoriality-cohmack}) remain cohomological Mackey functors.

These operations satisfy analogues of the axioms for Mackey functors, and this follows essentially from the corresponding axioms for permutation modules.
As an example, we give details for the Mackey formula.
Let $H,K\le \Gamma$ and $M\in\CohMack[\CR](K)$.
Then we find
\begin{align*}
  \rho^{\Gamma}_H\tau^{\Gamma}_KM
  &=\tau^{\Gamma}_KM\circ \Ind^{\Gamma}_H\\
  &=M\circ\Res^{\Gamma}_K\circ\Ind^{\Gamma}_H\\
  &=M\circ \oplus_{[g]\in K\backslash{}{\Gamma}/H}\Ind^K_{\Hg}\circ c_{g}\circ\Res^H_{\Kg}\\
  &=\oplus_{[g]\in K\backslash{}{\Gamma}/H}M\circ \Ind^K_{\Hg}\circ c_{g}\circ\Res^H_{\Kg}\\
  &=\oplus_{[g]\in K\backslash{}{\Gamma}/H}\tau^H_{\Kg}(M\circ \Ind^K_{\Hg}\circ c_{g})\\
  &=\oplus_{[g]\in K\backslash{}{\Gamma}/H}\tau^H_{\Kg}\sigma_{g\inv}(M\circ \Ind^K_{\Hg})\\
  &=\oplus_{[g]\in K\backslash{}{\Gamma}/H}\tau^H_{\Kg}\sigma_{g\inv}\rho^K_{\Hg}M\\
  &=\oplus_{[\gamma]\in H\backslash{}{\Gamma}/K}\tau^H_{{}^\gamma K\cap H}\sigma_{\gamma}\rho^K_{K\cap H^\gamma}M
\end{align*}
where the last equality is obtained upon replacing $g$ by $\gamma=g\inv$.
\end{Rem}

\begin{Rem}
\label{mackey-functoriality-Dcohmack}%
The operations $\rho^?_?$, $\tau^?_?$, $\sigma_?$ on cohomological Mackey functors are exact and therefore trivially induce functors (denoted by the same symbols) on the derived categories.
Indeed, induction and restriction are adjoints to each other on both sides, at the level of permutation modules.
It follows that the same is true for $\rho^?_?$ and $\tau^?_?$ at the level of cohomological Mackey functors.
In particular, these functors are exact.
The functor $\sigma_\gamma$ is an isomorphism with inverse $\sigma_{\gamma\inv}$.
\end{Rem}

\begin{Prop}
\label{mackey-functoriality-DPerm-D(CohMack)}%
Let $\Gamma'\le\Gamma$ be an open subgroup, and $\gamma\in\Gamma$.
The equivalence of~\Cref{Cor:DPerm-comack} identifies the operations on the left with the operations on the right in the following diagram:
\[
\begin{tikzcd}[column sep=large]
\DPerm(\Gamma;\CR)
\ar[r, "\sim"]
\ar[d, "\Res^\Gamma_{\Gamma'}" right, shift left=1em]
&
\D(\CohMack[\CR]{\Gamma})
\ar[d, "\rho^\Gamma_{\Gamma'}" right, shift left=1em]
\\
\DPerm(\Gamma';\CR)
\ar[u, "\Ind^\Gamma_{\Gamma'}" left, shift left=1em]
\ar[r, "\sim"]
\ar[d, "c_{\gamma}"]
&
\D(\CohMack[\CR](\Gamma'))
\ar[u, "\tau^\Gamma_{\Gamma'}" left, shift left=1em]
\ar[d, "\sigma_{\gamma}"]
\\
\DPerm({}^\gamma\Gamma';\CR)
\ar[r, "\sim"]
&
\D(\CohMack[\CR]({}^\gamma\Gamma'))
\end{tikzcd}
\]
\end{Prop}
\begin{proof}
Suppose $G:\MMod{H;\CR}\to\MMod{K;\CR}$ is a functor with left adjoint $F$.
Then for any $(H;\CR)$-module $M$ we have
\begin{align*}
  \FP(M)\circ F=\Hom_{\MMod{K;\CR}}(F(-),M)=\Hom_{\MMod{H;\CR}}(-,G(M))=\FP(G(M)).
\end{align*}
By the existing adjunctions for restriction, induction, and conjugation, and by our definition of the operations on cohomological Mackey functors, we see from this that the functor
\[
\FP:\MMod{?;\CR}\to\CohMack[\CR](?)
\]
is compatible with the operations in the statement.
The same is then true for the induced functor
\[
\FP:\K(\MMod{?;\CR})\to\K(\CohMack[\CR](?)).
\]
We conclude after restricting to $\DPerm(?;\CR)$ and composing with $\K(\CohMack[\CR](?))\to\D(\CohMack[\CR](?))$, by \Cref{mackey-functoriality-Dcohmack}.
\end{proof}

We now turn our focus to Artin motives.
The following functoriality is a special case of the adjunction $f_\sharp\dashv f^*$ at the level of effective motives, for any smooth morphism $f$ of schemes.
\begin{Rem}
\label{Rem:motives-smooth-functoriality}%
Let $\FF\subset\FF'$ be a finite separable field extension and let us denote by $\pi_{\FF'/\FF}:\Spec(\FF')\to\Spec(\FF)$ the associated \'etale morphism of schemes.
The scalar extension functor $\FF'\times_\FF-$ admits a left adjoint
\begin{equation}
\label{eq:etale-adjunction-corr}%
\pi_{\FF'/\FF}\circ -:\Corr_{\FF'}\rightleftarrows\Corr_{\FF}:\FF'\times_\FF-
\end{equation}
which takes a smooth $\FF'$-scheme $X$ to itself viewed as a smooth $\FF$-scheme~\cite[Lemma~9.3.7]{cisinski-deglise:DM}.
Left Kan extension and sheafification induce a similar adjunction on sheaves with transfers:
\[
(\pi_{\FF'/\FF})_\sharp:\Sh{\Nis}{\Corr_{\FF'}}{\CR}\rightleftarrows\Sh{\Nis}{\Corr_{\FF}}{\CR}:\pi_{\FF'/\FF}^*\,.
\]
Explicitly, $\pi_{\FF'/\FF}^*(F)=F\circ (\pi_{\FF'/\FF}\circ-)$, while $(\pi_{\FF'/\FF})_\sharp$ is essentially determined by sending the sheaf with transfers represented by $X\in\Corr_{\FF'}$ to the sheaf with transfers represented by $X$ viewed in $\Corr_{\FF}$.
The pullback functor is exact and passes to the derived category, where its left adjoint is given by left deriving $(\pi_{\FF'/\FF})_\sharp$.
By the description given above, it is clear that these two functors pass to an adjunction on the quotients~\eqref{eq:DMeff}, which in line with the literature we abusively denote by
\[
(\pi_{\FF'/\FF})_\sharp:\DMbigeff(\FF';\CR)\rightleftarrows\DMbigeff(\FF;\CR):\pi_{\FF'/\FF}^*\,.
\]
It is also true, although we will not need it, that the functor $\pi_{\FF'/\FF}^*$ admits a right adjoint $(\pi_{\FF'/\FF})_*$.
\end{Rem}

\begin{Rem}
\label{Rem:artin-motives-etale-functoriality}%
The adjunction of \Cref{Rem:motives-smooth-functoriality} clearly restricts to the subcategories of Artin motives, since on motives of smooth schemes they coincide with the adjunction of~\eqref{eq:etale-adjunction-corr}.
Thus we obtain
\[
(\pi_{\FF'/\FF})_\sharp:\DAMbig(\FF';\CR)\rightleftarrows \DAMbig(\FF;\CR):\pi_{\FF'/\FF}^*.
\]
We also note that the adjunction restricts to geometric Artin motives for the same reason.
\end{Rem}

\begin{Rem}
\label{Rem:Artin-motives-conjugation}%
Let $\FF$ be a field with a fixed separable algebraic closure $\FFsep$ and absolute Galois group $\Gamma=\Gal(\FFsep/\FF)$.
For any element $\gamma\in\Gamma$ and finite field extension $\FF\subseteq\FF'$ contained in $\FFsep$ we obtain an equivalence
$\gamma:\Sm^0_{\gamma\inv(\FF')}\isoto\Sm^0_{\FF'}$ which takes the spectrum of $\KK$ (for $\gamma\inv(\FF')\subset\KK\subset\FFsep$) to the spectrum of $\gamma(\KK)$.
We denote by the same symbol $\gamma:\Corr^0_{\gamma\inv(\FF')}\isoto\Corr^0_{\FF'}$ the associated equivalence on correspondences, and by $\gamma^*=-\circ \gamma$ the induced equivalence on sheaves with transfers.
Since it is exact, it passes to the derived category.
Through the equivalence of \Cref{sheaves-transfers-mackey} we obtain an adjoint equivalence
\[
\gamma^*:\DAMbig(\FF';\CR)\simeq\DAMbig(\gamma\inv(\FF');\CR):(\gamma\inv)^*.
\]
\end{Rem}

\begin{Prop}
\label{mackey-comack-artin}%
Let $\FF\subset\FF'$ be a finite separable extension, let $\FFsep$ be a separable algebraic closure of $\FF$ with Galois group $\Gamma=\Gal(\FFsep/\FF)$.
Assume that $\FF'\subset\FFsep$ corresponds to the open subgroup $\Gamma'\le\Gamma$.
Also let $\gamma\in\Gamma$.

The equivalences of \Cref{sheaves-transfers-mackey} and \Cref{DPerm-DAM} identify the operations on the left with the operations on the right in the following diagram:
\[
\begin{tikzcd}[column sep=large]
\D(\CohMack[\CR]{\Gamma})
\ar[d, "\rho^\Gamma_{\Gamma'}" right, shift left=1em]
\ar[r, "\sim"]
&
\DAMbig(\FF;\CR)
\ar[d, "\pi_{\FF'/\FF}^*" right, shift left=1em]
\\
\D(\CohMack[\CR]{\Gamma'})
\ar[u, "\tau^\Gamma_{\Gamma'}" left, shift left=1em]
\ar[r, "\sim"]
\ar[d, "\sigma_{\gamma}"]
&
\DAMbig(\FF';\CR)
\ar[u, "(\pi_{\FF'/\FF})_\sharp" left, shift left=1em]
\ar[d, "\gamma^*"]
\\
\D(\CohMack[\CR]{{}^\gamma\Gamma'})
\ar[r, "\sim"]
&
\DAMbig(\gamma\inv(\FF');\CR)
\end{tikzcd}
\]
\end{Prop}
\begin{proof}
We start with conjugation, that is, the bottom square.
If $\gamma\inv(\FF')\subseteq\KK\subseteq\FFsep$ corresponds to $K\le (\Gamma')^\gamma$ then $\gamma(\KK)$ corresponds to ${}^\gamma K\le \Gamma'$.
Under the equivalence between categories of correspondences and permutation modules of \Cref{Corr-perm-equivalence}, the operation $\gamma$ is therefore induced by $\CR((\Gamma')^\gamma/H)\mapsto\CR(\Gamma'/{}^\gamma H)$, that is, $c_{\gamma\inv}$.
By definition (\Cref{Rem:mackey-functoriality-cohmack,Rem:Artin-motives-conjugation}), it then follows that the operations $\sigma_\gamma=-\circ c_{\gamma\inv}$ and $\gamma^*=-\circ\gamma$ correspond to each other under the equivalence of \Cref{sheaves-transfers-mackey}.
Passing to the derived categories we see that the bottom square in the statement commutes.

For the top square, by uniqueness of adjoints, it suffices to discuss restriction.
Now, given a $0$-dimensional smooth $\FF$-scheme $X$, it is clear that the corresponding $\Gamma$-set is in bijection with the $\Gamma'$-set corresponding to the $\FF'$-scheme $\FF'\times_\FF X$:
\[
\Hom_{\FF'}(\FF'\otimes_\FF\mathcal{O}(X),\FFsep)\cong\Hom_{\FF}(\mathcal{O}(X),\FFsep).
\]
It follows that the functor $\FF'\times_\FF-$ on correspondences~\eqref{eq:etale-adjunction-corr} identifies with restriction on permutation modules.
Left Kan extending we obtain that the functor $\pi^*_{\FF'/\FF}$ on sheaves with transfers identifies with the left Kan extension of restriction on cohomological Mackey functors, under the equivalence of \Cref{sheaves-transfers-mackey}.
The latter is left adjoint to the functor $\tau^\Gamma_{\Gamma'}=-\circ\Res^\Gamma_{\Gamma'}$ of~\eqref{eq:mackey-functoriality-cohmack}, thus coincides with $\rho^\Gamma_{\Gamma'}$.
Passing to derived categories, we conclude that the top square(s) commute(s) as well.
\end{proof}

Combining \Cref{mackey-functoriality-DPerm-D(CohMack),mackey-comack-artin} we obtain that the Mackey functoriality on the derived category of permutation modules identifies with the one on Artin motives.
\begin{Cor}
With the notation and assumptions of \Cref{mackey-comack-artin}, the equivalence of \Cref{DPerm-DAM} identifies the operations on the left with the operations on the right in the following diagram:
\[
\pushQED{\qed}
\adjustbox{valign=b}{\begin{tikzcd}[column sep=large]
\DPerm(\Gamma;\CR)
\ar[r, "\sim"]
\ar[d, "\Res^\Gamma_{\Gamma'}" right, shift left=1em]
&
\DAMbig(\FF;\CR)
\ar[d, "\pi_{\FF'/\FF}^*" right, shift left=1em]
\\
\DPerm(\Gamma';\CR)
\ar[u, "\Ind^\Gamma_{\Gamma'}" left, shift left=1em]
\ar[r, "\sim"]
\ar[d, "c_{\gamma}"]
&
\DAMbig(\FF';\CR)
\ar[u, "(\pi_{\FF'/\FF})_\sharp" left, shift left=1em]
\ar[d, "\gamma^*"]
\\
\DPerm({}^\gamma\Gamma';\CR)
\ar[r, "\sim"]
&
\DAMbig(\gamma\inv(\FF');\CR)
\end{tikzcd}
}
\qedhere
\popQED
\]
\end{Cor}

We finish with an observation about inflation.
Let $N\subseteq \Gamma$ be a \emph{closed but not necessarily open} normal subgroup, corresponding to a Galois extension $\FF\subseteq\EE\subseteq\FFsep$.
\begin{Def}
\label{Def:bounded-DAM}%
We denote by $\DAMbig(\EE/\FF;\CR)$ the localizing subcategory of $\DAMbig(\FF;\CR)$ generated by the motives of $\Spec(\FF')$, where $\FF\subseteq\FF'\subseteq\EE$ is a finite subextension of~$\FF$.
\end{Def}

\begin{Prop}
\label{Prop:Infl-Inclusion}%
The equivalence of \Cref{DPerm-DAM} identifies the subcategory on the left with the subcategory on the right in the following diagram:
\[
\begin{tikzcd}[column sep=large]
\DPerm(\bar\Gamma;\CR)
\ar[r, "\sim"]
\arrow[d, rightarrowtail, "\Infl_{\bar{\Gamma}}^{\Gamma}" right]
&
\DAMbig(\EE/\FF;\CR)
\ar[d, rightarrowtail]
\\
\DPerm(\Gamma;\CR)
\ar[r, "\sim"]
&
\DAMbig(\FF;\CR).
\end{tikzcd}
\]
\end{Prop}
\begin{proof}
The vertical functor on the left is fully faithful, by \Cref{Lem:Infl-fully-faithful-DPerm}.
The claim then follows from the observation that the composite of \Cref{DPerm-DAM} with inflation takes the compact generators $\CR(\bar\Gamma/\bar K)$ of $\DPerm(\bar\Gamma;\CR)$ defined by $N\subseteq K\le \Gamma$ to the compact generators of $\DAMbig(\EE/\FF;\CR)$, namely the motive of $\Spec(\FFsep^K)$.
\end{proof}

\begin{Rem}
Note that $\DAM(\FF;\CR)$ is the union of $\DAM(\EE/\FF;\CR)$ where $\EE$ runs through \emph{finite} Galois extensions, and where $\DAM(\EE/\FF;\CR)=\DAMbig(\EE/\FF;\CR)\cap\DAM(\FF;\CR)$ is the compact part.
This corresponds, by \Cref{Prop:Infl-Inclusion}, to the remark made about the derived category of permutation modules in \Cref{Rem:Kbperm-union}.
\end{Rem}


\bibliographystyle{alpha}
\bibliography{ref}

\end{document}